\documentclass[11pt,draftcls,onecolumn]{IEEEtran}

\usepackage{macros}

\hypersetup{
  colorlinks = true,
  urlcolor   = red,
  linkcolor  = blue,
  citecolor  = red
}

\makeatletter

\@addtoreset{theorem}{section}
\@addtoreset{proposition}{section}
\@addtoreset{corollary}{section}
\@addtoreset{example}{section}
\@addtoreset{remark}{section}
\@addtoreset{lemma}{section}
\@addtoreset{definition}{section}
\@addtoreset{question}{section}
\@addtoreset{conjecture}{section}

\def\@seccntformat#1{\csname the#1\endcsname\quad}
\makeatother

\begin{document}
\setlength{\baselineskip}{1\baselineskip}

\title{\Huge{On Strongly Regular Graphs and the Friendship Theorem}}

\author{Igal Sason
\thanks{
Igal Sason is with the Viterbi Faculty of Electrical and Computer Engineering and the Department of Mathematics (as a
secondary appointment) at the Technion --- Israel Institute of Technology, Haifa 3200003, Israel. Email: sason@ee.technion.ac.il.}}

\maketitle

\thispagestyle{empty}
\setcounter{page}{1}

\vspace*{-1cm}

{\footnotesize \noindent {\bf Abstract.}
This paper presents an alternative proof of the celebrated friendship theorem, originally established by Erd\H{o}s, R\'{e}nyi,
and S\'{o}s (1966). The proof relies on a closed-form expression for the Lov\'{a}sz $\vartheta$-function of strongly regular graphs,
recently derived by the author. Additionally, the paper considers some known extensions of the theorem, offering discussions that provide
insights into the friendship theorem, one of its extensions, and the proposed proof.
Leveraging the closed-form expression for the Lov\'{a}sz $\vartheta$-function of strongly regular graphs, the paper further establishes
new necessary conditions for a strongly regular graph to be a spanning or induced subgraph of another strongly regular graph.
In the case of induced subgraphs, the analysis also incorporates a property of graph energies. Some of these results are extended
to regular graphs and their subgraphs.}

\vspace*{0.1cm}

\vspace*{0.1cm}
\noindent {\bf 2020 Mathematics Subject Classification (MSC).} 05C50, 05C60, 05E30.

\section{Introduction}
\label{section: Introduction}

The friendship theorem in graph theory states that if a finite graph has the property that every pair of distinct vertices
shares exactly one common neighbor, then the graph consists of edge-disjoint triangles sharing a single, central vertex.
In other words, such a graph is necessarily a friendship graph (also called a windmill graph), where a central vertex is
adjacent to all others, and additional edges form edge-disjoint triangles.
The theorem's name reflects its intuitive human interpretation, which asserts that if every two individuals in a (finite) group
have exactly one mutual friend, then there must be someone who is everybody's friend (often referred to as the "politician").

That theorem was first proved by Erd\H{o}s, R\'{e}nyi, and S\'{o}s \cite{ErdosRS66}, and it found diverse applications
in combinatorics and other disciplines. While the friendship theorem itself is a result in graph theory, its structural properties
make it useful in designing graph-based codes, network coding, and cryptographic schemes. In network coding, graphs represent
information flow between nodes, and the friendship property ensures that all transmissions go through a single relay node, which
models centralized communication structures such as satellite communication, where every ground station interacts via a single hub.
This property is also useful in wireless sensor networks, where multiple sensors communicate through a central processing node. Furthermore,
friendship graphs can be linked to block designs and combinatorial structures used to construct error-correcting codes with good
covering properties. Beyond these applications, the friendship theorem has a broad range of applications in computer science, biology,
social sciences, and economics. The fact that every two vertices share exactly one mutual neighbor makes it a useful model for systems
where a single hub controls multiple independent interactions, such as in social influence, network protocols, and centralized
coordination problems.

The original proof of the theorem involves arguments based on graph connectivity, degree counting, and combinatorial reasoning.
The reader is referred to a nice exposition of a proof of the friendship theorem in Chapter~44 of \cite{AignerZ18}, which combines
combinatorics and linear algebra (spectral graph theory). Due to its importance and simplicity, the theorem has been proved in various ways.
For historical insights into these proofs, as well as a novel combinatorial proof and an extension of the theorem, the reader is referred to
\cite{MertziosU16}.

The aim of this paper is twofold. First, it provides an alternative proof of that theorem, relying on a recent result by
the author \cite{Sason23}, in which a closed-form expression for the celebrated Lov\'{a}sz $\vartheta$-function was derived for the
structured family of strongly regular graphs. We believe this proof offers a new perspective on the existing collection of proofs
of the friendship theorem. The paper further considers some known extensions of the theorem, offering discussions that provide insights
into the friendship theorem, one of its extensions from \cite{MertziosU16}, and the proposed proof. The second aim of this paper,
building on the closed-form expression for the Lov\'{a}sz $\vartheta$-function of all strongly regular graphs, is to establish new necessary
conditions for a strongly regular graph to be a spanning or induced subgraph of another strongly regular graph. For induced subgraphs,
the analysis also incorporates a property of graph energies. Some of these results are extended to regular graphs and their subgraphs.
In general, determining whether a graph is a spanning or induced subgraph of another given graph is an important and broadly applicable
problem across theoretical, algorithmic, and applied areas of graph theory.

This paper is structured as follows. Section~\ref{section: preliminaries} introduces the notation and essential background
required for the analysis in this paper. Section~\ref{section: friendship theorem, proofs, and extensions} provides our alternative
proof of the friendship theorem, followed by a variation of a known spectral graph-theoretic proof, and a consideration of
the friendship theorem and our alternative proof to gain further insights into the theorem, our proposed proof, and one of
the theorem's extensions in \cite{MertziosU16}.
Based on the Lov\'{a}sz $\vartheta$-function of strongly regular graphs, Sections~\ref{section: spanning subgraphs}
and~\ref{section: induced subgraphs} establish new necessary conditions for one strongly regular graph to be a spanning
or induced subgraph, respectively, of another strongly regular graph. The proposed conditions can be easily checked,
their utility is demonstrated, and they are extended to regular graphs and subgraphs, with a consideration of their
computational complexity.

\section{Preliminaries}
\label{section: preliminaries}

This section presents the notation and necessary background, including definitions and theorems
essential for the analysis in this paper.

Let $\Gr{G}$ be a graph with vertex set $\V{\Gr{G}}$ and edge set $\E{\Gr{G}}$.
A graph $\Gr{G}$ is called simple if it has no self-loops and no multiple edges between any pair
of vertices. Throughout the paper, unless explicitly mentioned, all graphs under consideration are
simple, finite, and undirected. The following standard notation and definitions are used.

\begin{definition}[Graph complement]
{\em The complement of a graph $\Gr{G}$, denoted by $\CGr{G}$, is a graph whose vertex set is
$\V{\Gr{G}}$, and its edge set is the complement set $\CGr{\E{\Gr{G}}}$.
Every vertex in $\V{\Gr{G}}$ is nonadjacent to itself in $\Gr{G}$ and $\CGr{G}$, so $\{i,j\} \in \E{\CGr{G}}$
if and only if $\{i, j\} \notin \E{\Gr{G}}$ with $i \neq j$.}
\end{definition}

\begin{definition}[Integer-valued graph invariants]
\noindent
{\em
\begin{itemize}
\item Let $k \in \naturals $. A proper $k$-coloring of a graph $\Gr{G}$
is a function $c \colon \V{\Gr{G}} \to \{1,2,...,k\}$, where $c(v) \ne c(u)$ for every
$\{u,v\}\in \E{\Gr{G}}$. The chromatic number of $\Gr{G}$, denoted by $\chi(\Gr{G})$,
is the smallest $k$ for which there exists a proper $k$-coloring of $\Gr{G}$.
\item A clique in a graph $\Gr{G}$ is a subset of vertices $U\subseteq \V{\Gr{G}}$,
where $\{u,v\} \in \E{\Gr{G}}$ for every $u,v \in U$ with $u \neq v$. The clique number
of $\Gr{G}$, denoted by $\omega(\Gr{G})$, is the largest size of a clique in $\Gr{G}$.
\item An independent set in a graph $\Gr{G}$ is a subset of vertices $U\subseteq \V{\Gr{G}}$,
where $\{u,v\} \notin \E{\Gr{G}}$ for every $u,v \in U$. The independence number of $\Gr{G}$,
denoted by $\indnum{\Gr{G}}$, is the largest size of an independent set in $\Gr{G}$.
Consequently, $\indnum{\Gr{G}} = \omega(\CGr{G})$ for every graph $\Gr{G}$.
\end{itemize}
These integer-valued functions of a graph are invariant under graph isomorphisms, so they
are referred to as graph invariants.}
\end{definition}

We next introduce the Lov\'{a}sz $\vartheta$-function of a graph $\Gr{G}$, and then consider
some of its useful properties. To that end, orthonormal representations of graphs are introduced.
\begin{definition}[Orthonormal representations of a graph]
\label{def: orthonormal representation}
{\em Let $\Gr{G}$ be a finite, simple, and undirected graph, and $d \in \naturals$.
An orthonormal representation of $\Gr{G}$ in the $d$-dimensional Euclidean space $\Reals^d$
assigns to each vertex $i \in \V{\Gr{G}}$ a unit vector ${\bf{u}}_i \in \Reals^d$ such that,
for every two distinct and nonadjacent vertices in $\Gr{G}$, their assigned vectors are orthogonal.}
\end{definition}
\begin{definition}[Lov\'{a}sz $\vartheta$-function, \cite{Lovasz79_IT}]
\label{definition: Lovasz function}
{\em Let $\Gr{G}$ be a finite, undirected and simple graph.
The Lov\'{a}sz $\vartheta$-function of $\Gr{G}$ is defined as
\begin{eqnarray}
\label{eq: Lovasz theta function}
\vartheta(\Gr{G}) \triangleq \min_{\{\bf{u}_i\}, \bf{c}} \; \max_{i \in \V{\Gr{G}}} \,
\frac1{\bigl( {\bf{c}}^{\mathrm{T}} {\bf{u}}_i \bigr)^2} \, ,
\end{eqnarray}
where the minimum is taken over all orthonormal representations $\{{\bf{u}}_i: i \in \V{\Gr{G}} \}$ of $\Gr{G}$, and
over all unit vectors ${\bf{c}}$.
The unit vector $\bf{c}$ is called the {\em handle} of the orthonormal representation. By the Cauchy-Schwarz inequality,
$\bigl|{\bf{c}}^{\mathrm{T}} {\bf{u}}_i \bigr| \leq \|{\bf{c}} \| \, \|{\bf{u}}_i \| = 1$, so
$\vartheta(\Gr{G}) \geq 1$, with equality if and only if $\Gr{G}$ is a complete graph.}
\end{definition}
The dimension $d$ of the Euclidean space $\Reals^d$ in the orthonormal representations of
$\Gr{G}$, over which the minimization on the right-hand side of \eqref{eq: Lovasz theta function} is performed can be
set to the order of $\Gr{G}$, i.e., $d = \card{\V{\Gr{G}}}$ (see p.~183 of \cite{Lovasz19}).
Let the following notation be used:
\begin{itemize}
\item ${\bf{A}}$ is the $n \times n$ adjacency matrix of $\Gr{G}$ \, ($n \triangleq \card{\V{\Gr{G}}}$), where $A_{i,j}=1$
if and only if $\{i,j\} \in \E{\Gr{G}}$ and $A_{i,j}=0$ otherwise;
\item ${\bf{J}}_n$ is the all-ones $n \times n$ matrix;
\item ${\bf{\set{S}}}_{+}^n$ is the set of all $n \times n$ positive semidefinite matrices;
\item
The eigenvalues of ${\bf{A}}$ are given in decreasing order by
\begin{eqnarray}
\Eigval{\max}{\Gr{G}} = \Eigval{1}{\Gr{G}} \geq \Eigval{2}{\Gr{G}}
\geq \ldots \geq \Eigval{n}{\Gr{G}} = \Eigval{\min}{\Gr{G}}.
\end{eqnarray}

\item
The {\em adjacency spectrum} of $\Gr{G}$ is the multiset of the eigenvalues
of ${\bf{A}}$, counted with multiplicities.
\end{itemize}

The following semidefinite program (SDP) computes $\vartheta(\Gr{G})$ (by Theorem~4 of \cite{Lovasz79_IT}):
\begin{equation}
\label{eq: SDP}
\mbox{\fbox{$
\begin{array}{l}
\text{maximize} \; \; \text{Trace}({\bf{B}} \, {\bf{J}}_n) \\
\text{subject to} \\
\begin{cases}
{\bf{B}} \in {\bf{\set{S}}}_{+}^n, \; \; \text{Trace}({\bf{B}}) = 1, \\
A_{i,j} = 1  \; \Rightarrow \;  B_{i,j} = 0, \quad i,j \in \OneTo{n}.
\end{cases}
\end{array}$}}
\end{equation}
This renders the computational complexity of $\vartheta(\Gr{G})$ feasible. Specifically,
there exist standard algorithms in convex optimization that numerically compute $\vartheta(\Gr{G})$,
for every graph $\Gr{G}$, with precision of $r$ decimal digits, and in polynomial-time in
$n$ and $r$ (see Section~11.3 of \cite{Lovasz19}).

The Lov\'{a}sz $\vartheta$-function
of $\Gr{G}$ and its complement graph $\CGr{G}$ satisfy the following \cite{Lovasz79_IT,Lovasz19}:
\begin{enumerate}[(1)]
\item Sandwich theorem:
\vspace*{-0.6cm}
\begin{eqnarray}
\label{eq1a: sandwich}
& \indnum{\Gr{G}} \leq \vartheta(\Gr{G}) \leq \chi(\CGr{G}), \\[0.1cm]
\label{eq1b: sandwich}
& \omega(\Gr{G}) \leq \vartheta(\CGr{G}) \leq \chi(\Gr{G}).
\end{eqnarray}
\item Computational complexity:
The graph invariants $\indnum{\Gr{G}}$, $\omega(\Gr{G})$, and $\chi(\Gr{G})$ are NP-hard problems.
However, as mentioned above, the numerical computation of $\vartheta(\Gr{G})$ is in general feasible,
so it provides polynomial-time computable bounds on very useful graph invariants that are hard to compute.
\item Hoffman-Lov\'{a}sz inequality: Let $\Gr{G}$ be $d$-regular
of order $n$. Then,
\begin{eqnarray}
\label{eq: Lovasz79 - Theorem 9}
\vartheta(\Gr{G}) \leq -\frac{n \, \Eigval{n}{\Gr{G}}}{d - \Eigval{n}{\Gr{G}}},
\end{eqnarray}
with equality if $\Gr{G}$ is edge-transitive.
\end{enumerate}

\begin{definition}[Strongly regular graphs]
\label{definition: srg}
{\em Let $\Gr{G}$ be a $d$-regular graph of order $n$. The graph $\Gr{G}$
is a {\em strongly regular} graph (SRG) if there exist nonnegative
integers $\lambda$ and $\mu$ such that
\begin{itemize}
\item Every pair of adjacent vertices has exactly $\lambda$ common neighbors;
\item Every pair of distinct, non-adjacent vertices has exactly $\mu$ common neighbors.
\end{itemize}
Such a strongly regular graph is said to be a graph in the family $\srg{n}{d}{\lambda}{\mu}$.}
\end{definition}
It is important to note that these four parameters are interrelated (see
Proposition~\ref{proposition: Feasible Parameters of Strongly Regular Graphs} for known
results). Consequently, for some parameter vectors $(n,d,\lambda,\mu)$, the family
$\srg{n}{d}{\lambda}{\mu}$ contains no graphs.
Furthermore, for parameter vectors $(n, d, \lambda, \mu)$ such that the set
$\srg{n}{d}{\lambda}{\mu}$ is nonempty, there may exist several nonisomorphic strongly
regular graphs within it. For example, according to \cite{Brouwer}, there
exist 167~nonisomorphic strongly regular graphs in the family $\srg{64}{18}{2}{6}$.
The reader is referred to \cite{Brouwer,Maksimovic18,BrouwerM22} for the enumeration of
strongly regular graphs with given parameter vectors.

\begin{proposition} {\em \cite{BrouwerM22}}
\label{proposition: srg complement}
{\em A graph $\Gr{G}$ is strongly regular if and only if its complement $\CGr{G}$ is so. Furthermore,
if $\Gr{G}$ is a strongly regular graph in the family $\srg{n}{d}{\lambda}{\mu}$, then
$\CGr{G}$ is a strongly regular graph in the family $\srg{n}{n-d-1}{n-2d+\mu-2}{n-2d+\lambda}$.}
\end{proposition}

\begin{theorem}[Bounds on Lov\'{a}sz function of regular graphs, \cite{Sason23}]
\label{theorem: vartheta srg}
{\em Let $\Gr{G}$ be a $d$-regular graph of order $n$, which is a non-complete
and non-empty graph. Then, the following bounds hold for the Lov\'{a}sz
$\vartheta$-function of $\Gr{G}$ and its complement $\CGr{G}$:
\begin{enumerate}[(1)]
\item
\begin{eqnarray}
\label{eq:21.10.22a1}
\frac{n-d+\Eigval{2}{\Gr{G}}}{1+\Eigval{2}{\Gr{G}}} \leq \vartheta(\Gr{G})
\leq -\frac{n \Eigval{n}{\Gr{G}}}{d - \Eigval{n}{\Gr{G}}}.
\end{eqnarray}

\begin{itemize}
\item Equality holds in the leftmost inequality of \eqref{eq:21.10.22a1} if $\CGr{G}$
is both vertex-transitive and edge-transitive, or if $\Gr{G}$ is a strongly regular graph;
\item Equality holds in the rightmost inequality of \eqref{eq:21.10.22a1} if $\Gr{G}$
is edge-transitive, or if $\Gr{G}$ is a strongly regular graph.
\end{itemize}
\item
\begin{eqnarray}
\label{eq:21.10.22a2}
1 - \frac{d}{\Eigval{n}{\Gr{G}}} \leq \vartheta(\CGr{G})
\leq \frac{n \bigl(1+\Eigval{2}{\Gr{G}}\bigr)}{n-d+\Eigval{2}{\Gr{G}}}.
\end{eqnarray}
\begin{itemize}
\item Equality holds in the leftmost inequality of \eqref{eq:21.10.22a2}
if $\Gr{G}$ is both vertex-transitive and edge-transitive, or if $\Gr{G}$ is
a strongly regular graph;
\item Equality holds in the rightmost inequality of \eqref{eq:21.10.22a2}
if $\CGr{G}$ is edge-transitive, or if $\Gr{G}$ is a strongly regular graph.
\end{itemize}
\end{enumerate}}
\end{theorem}

As a common sufficient condition, note that all the inequalities in \eqref{eq:21.10.22a1}
and \eqref{eq:21.10.22a2} hold with equality if $\Gr{G}$ is a strongly regular graph.
The following result provides a closed-form expression of the Lov\'{a}sz $\vartheta$-function of
all strongly regular graphs.

\begin{theorem}[The Lov\'{a}sz $\vartheta$-function of strongly regular graphs, \cite{Sason23}]
\label{theorem: theta of srg}
{\em Let $\Gr{G}$ be a strongly regular graph in the family $\SRG(n, d, \lambda, \mu)$.
Then,
\begin{align}
& \vartheta(\Gr{G}) = \dfrac{n \, (t+\mu-\lambda)}{2d+t+\mu-\lambda}, \label{eq1: 15.02.25}  \\[0.2cm]
& \vartheta(\CGr{G}) = 1 + \dfrac{2d}{t+\mu-\lambda},   \label{eq2: 15.02.25}
\end{align}
where
\begin{align}
t \triangleq \sqrt{(\mu-\lambda)^2 + 4(d-\mu)}.  \label{eq3: 15.02.25}
\end{align}}
\end{theorem}
\begin{remark}
{\em In light of Theorem~\ref{theorem: theta of srg}, all strongly regular graphs in
a family $\srg{n}{d}{\lambda}{\mu}$ are not only cospectral \cite{BrouwerM22},
but they also share the same value of the Lov\'{a}sz $\vartheta$-function.}
\end{remark}

\begin{corollary}[\cite{Sason23}]
\label{corollary: property - SRGs or v.t.}
{\em Let $\Gr{G}$ be a strongly regular graph on $n$ vertices. Then,
\begin{align}
\label{eq: property - SRGs or v.t.}
\vartheta(\Gr{G}) \, \vartheta(\CGr{G}) = n.
\end{align}}
\end{corollary}
\begin{remark}
{\em Corollary~\ref{corollary: property - SRGs or v.t.} shows that equality~\eqref{eq: property - SRGs or v.t.}
does not only for all vertex-transitive graphs (see Theorem~8 of \cite{Lovasz79_IT}), but it also holds for
all strongly regular graphs. It should be noted that strongly regular graphs may not be vertex-transitive
(see, e.g., \cite{Gritsenko01}).}
\end{remark}

The closed-form expression in Theorem~\ref{theorem: theta of srg} for the Lov\'{a}sz $\vartheta$-function
of strongly regular graphs eliminates the need for a numerical solution to the SDP in \eqref{eq: SDP},
while also providing an explicit analytical expression for that function. This facilitates the derivation of analytical bounds
on key graph invariants, which are generally NP-hard to compute. These bounds are expressed in terms of the parameters
$(n,d,\lambda,\mu)$ characterizing strongly regular graphs, as it is next shown in Corollary~\ref{corollary: bounds on graph invariants for SRGs}.
\begin{corollary}[Bounds on graph invariants of strongly regular graphs, \cite{Sason23}]
\label{corollary: bounds on graph invariants for SRGs}
{\em Let $\Gr{G}$ be a strongly regular graph in the family $\SRG(n, d, \lambda, \mu)$.
Then,
\begin{eqnarray}
& \indnum{\Gr{G}} \leq \bigg\lfloor \dfrac{n \, (t+\mu-\lambda)}{2d+t+\mu-\lambda} \bigg\rfloor \label{eq1: 25.02.2025} \\[0.1cm]
& \omega(\Gr{G}) \leq 1 + \bigg\lfloor \dfrac{2d}{t+\mu-\lambda} \bigg\rfloor, \label{eq2: 25.02.2025} \\[0.1cm]
& \chi(\Gr{G}) \geq 1 + \bigg\lceil \dfrac{2d}{t+\mu-\lambda} \bigg\rceil, \label{eq3: 25.02.2025} \\[0.1cm]
& \chi(\CGr{G}) \geq \bigg\lceil \dfrac{n \, (t+\mu-\lambda)}{2d+t+\mu-\lambda} \bigg\rceil,  \label{eq4: 25.02.2025}
\end{eqnarray}
where $t$ is given in \eqref{eq3: 15.02.25}.}
\end{corollary}
\begin{proof}
The bounds in \eqref{eq1: 25.02.2025}--\eqref{eq4: 25.02.2025} follow from the combination of the sandwich theorem
in \eqref{eq1a: sandwich} and \eqref{eq1b: sandwich} with Theorem~\ref{theorem: theta of srg}.
\end{proof}

\begin{example}[Bounds on graph invariants of strongly regular graphs]
{\em The tightness of the bounds in Corollary~\ref{corollary: bounds on graph invariants for SRGs} is exemplified
for four strongly regular graphs as follows.
\begin{enumerate}[(1)]
\item Petersen graph:
Let $\Gr{G}_1$ be the Petersen graph, the unique strongly regular graph in the family $\srg{10}{3}{0}{1}$ (see
Section~10.3 of \cite{BrouwerM22}). For $\Gr{G}_1$, the upper bounds on its independence and clique numbers in \eqref{eq1: 25.02.2025}
and \eqref{eq2: 25.02.2025}, respectively, as well as the lower bound on its chromatic number in \eqref{eq3: 25.02.2025} are tight:
\begin{align}
\indnum{\Gr{G}_1} = 4, \quad \omega(\Gr{G}_1) = 2, \quad \chi(\Gr{G}_1) = 3.
\end{align}
\item Schl\"{a}fli, Shrikhande, and Hall-Janko graphs:
\begin{itemize}
\item The Schl\"{a}fli graph is (up to isomorphism) the unique strongly regular graph in the family $\srg{27}{16}{10}{8}$ (see Section~10.7 of
\cite{BrouwerM22}).
\item The Shrikhande graph is one of two nonisomorphic strongly regular graphs in the family $\srg{16}{6}{2}{2}$ (see Section~10.6 of
\cite{BrouwerM22}).
\item The Hall-Janko graph is the unique strongly regular graph in the family $\srg{100}{36}{14}{12}$ (see Section~10.32
of \cite{BrouwerM22}).
\end{itemize}
Let $\Gr{G}_2$, $\Gr{G}_3$, and $\Gr{G}_4$ denote these graphs, respectively. The lower bounds
on the chromatic numbers of these graphs, as given in \eqref{eq3: 25.02.2025}, are all tight:
\begin{align}
\chi(\Gr{G}_2) = 9, \quad \chi(\Gr{G}_3) = 4, \quad \chi(\Gr{G}_4) = 10.
\end{align}
\item[(3)] For the Shrikhande graph ($\Gr{G}_3$),
\begin{itemize}
\item the upper bound on its independence number in \eqref{eq1: 25.02.2025} is tight: $\indnum{\Gr{G}_3} = 4$,
\item the upper bound on its clique number in \eqref{eq2: 25.02.2025} is, however, not tight as the upper bound is~4, while the
actual clique number is $\omega(\Gr{G}_3) = 3$.
\end{itemize}
\end{enumerate}}
\end{example}

\section{The Friendship Theorem: An Alternative Proof and Further Extensions}
\label{section: friendship theorem, proofs, and extensions}

This section is structured as follows: it presents the renowned friendship theorem, originally
established in \cite{ErdosRS66} (see Section~\ref{subsection: friendship theorem}), and it provides a new
alternative proof that relies on the Lov\'{a}sz $\vartheta$-function of strongly regular graphs (see
Theorem~\ref{theorem: theta of srg} and Section~\ref{subsection: alternative proof - Lovasz function}).
This section further provides a variation of a proof from Chapter~44 of \cite{AignerZ18}, relying on spectral properties of
strongly regular graphs (see Section~\ref{subsection: alternative proof - spectrum of srg}). It finally considers
further generalizations, providing discussions in order to get further insights into the theorem, our proposed proof,
and one of the theorem’s extensions in \cite{MertziosU16}, which are also supported by numerical results for (small
and large) strongly regular graphs (see Section~\ref{subsection:  extensions of the friendship theorem}).

\subsection{The friendship theorem}
\label{subsection: friendship theorem}

\begin{theorem}[Friendship Theorem, \cite{ErdosRS66}]
\label{thm: friendship theorem}
{\em Let $\Gr{G}$ be a finite graph in which any two distinct vertices have a single
common neighbor. Then, $\Gr{G}$ consists of edge-disjoint triangles sharing a single
vertex that is adjacent to every other vertex.}
\end{theorem}

A human interpretation of Theorem~\ref{thm: friendship theorem} is well known.
Assume that there is a party with $n$ people, where every pair of individuals has precisely one common friend at the party.
Theorem~\ref{thm: friendship theorem} asserts that one of these people is a friend of everyone.
Indeed, construct a graph whose vertices represent the $n$ people, and every two vertices are adjacent
if and only if they represent two friends. The claim then follows from Theorem~\ref{thm: friendship theorem}.

\begin{remark}[On Theorem~\ref{thm: friendship theorem}]
\label{remark: on the friendship theorem}
{\em The windmill graph (see Figure~\ref{fig:windmill_graph}) has the desired property of the friendship theorem
and it turns out to be the only graph with that property.
Notably, the friendship theorem does not hold for infinite graphs. Indeed, for an inductive construction
of a counterexample, one may start with a 5-cycle $\CG{5}$, and repeatedly add a common neighbor for every pair
of vertices that does not yet have one. This process results in a countably infinite friendship graph without
a vertex adjacent to all other vertices.}
\end{remark}

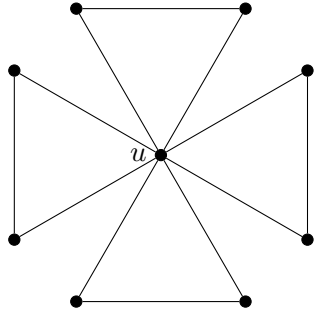
\begin{figure}[hbt]
    \centering
    \begin{tikzpicture}[scale=1.5, every node/.style={draw, circle, inner sep=1.5pt, fill=black}]
        \node[label=left:$u$] (u) at (0,0) {};

        \foreach \i in {0,1,2,3} {
            \node (a\i) at ({90*\i+30}:1.5) {};
            \node (b\i) at ({90*\i-30}:1.5) {};

            \draw (u) -- (a\i) -- (b\i) -- (u);
        }
    \end{tikzpicture}
    \caption{\centering{Windmill graph (also known as friendship graph) consisting of multiple triangles sharing a common central vertex $u$.}}
    \label{fig:windmill_graph}
\end{figure}

\subsection{A new alternative proof of the friendship theorem relying on the Lov\'{a}sz $\vartheta$-function}
\label{subsection: alternative proof - Lovasz function}
This subsection presents an alternative proof of Theorem~\ref{thm: friendship theorem},
utilizing the Lov\'{a}sz $\vartheta$-function of strongly regular graphs (and their complements)
in Theorem~\ref{theorem: theta of srg}. Recall that, by Proposition~\ref{proposition: srg complement},
the complements of strongly regular graphs are strongly regular as well.

\begin{proof}
Suppose the assertion is false, and let $\Gr{G}$ be a counterexample -- a finite graph in which any two distinct
vertices have a single common neighbor, yet no vertex in $\Gr{G}$ is adjacent to all other vertices.
A contradiction is obtained as follows: \newline
The first step shows that the graph $\Gr{G}$ is regular, as proved in Chapter~44 of \cite{AignerZ18}.
We provide a variation of that regularity proof, and then take a different approach, relying on graph invariants such as
$\omega(\Gr{G}), \chi(\Gr{G})$, and $\vartheta(\CGr{G})$.

\begin{itemize}
\item To assert the regularity of $\Gr{G}$, it is first proved that nonadjacent vertices in $\Gr{G}$ have equal
degrees, i.e., $d(u) = d(v)$ if $\{u,v\} \not\in \E{\Gr{G}}$.
\item The given hypothesis yields that $\Gr{G}$ is a connected graph (having a diameter of at most~2).
Let $\{u, v\} \not\in \E{\Gr{G}}$, and let $\Ng{u}$ and $\Ng{v}$ denote,
respectively, the sets of neighbors of the nonadjacent vertices $u$ and $v$.
\item Let $f \colon \Ng{u} \to \Ng{v}$ be the injective function
where every $x \in \Ng{u}$ is mapped to the unique $y \in \Ng{x} \cap \Ng{v}$.
Indeed, if $z \in \Ng{u} \setminus \{x\}$ satisfies $f(z) = y$, then $x$
and $z$ share two common neighbors (namely, $y$ and $u$), which contradicts the assumption of the theorem.
\item Since $f \colon \Ng{u} \to \Ng{v}$ is injective, it follows that
$\card{\Ng{u}} \leq \card{\Ng{v}}$.
\item By symmetry, swapping $u$ and $v$
(as nonadjacent vertices of the undirected graph $\Gr{G}$) also yields $\card{\Ng{v}} \leq \card{\Ng{v}}$,
so we get $d(u) = \card{\Ng{u}} = \card{\Ng{v}} = d(v)$ for all vertices $u, v \in \V{\Gr{G}}$ such that
$\{u, v\} \not\in \E{\Gr{G}}$.
\item To complete the proof that $\Gr{G}$ is regular, let $u$ and $v$ be fixed, nonadjacent vertices in $\Gr{G}$.
Consequently, $d(u) = d(v)$.
By the assumption of the theorem, except for one vertex, all other vertices are either nonadjacent to $u$ or $v$.
Hence, except for that vertex, all these vertices must have identical degrees by what we already proved.
\item Finally, by our further assumption (later leading to a contradiction), since
there is no vertex in $\Gr{G}$ that is adjacent to all other vertices, also the single vertex that is adjacent to $u$ and $v$
has a nonneighbor in $\Gr{G}$, so it also should have an identical degree to all the degrees of the other vertices
(by what is proved in the previous item). Consequently, $\Gr{G}$ is a regular graph.
\end{itemize}

From that point, our proof proceeds differently.
\begin{itemize}
\item Let $\Gr{G}$ be a $k$-regular graph on $n$ vertices. By the theorem's assumption, every two distinct vertices have exactly one common neighbor,
so $\Gr{G}$ is either a connected, strongly regular graph in the family $\srg{n}{k}{1}{1}$ if $k \geq 3$, or $\Gr{G} = \CoG{3}$ if $k=2$, or $\Gr{G} = \CoG{1}$
if $k=0$ (recall that complete graphs are excluded from the family of strongly regular graphs).
\item First, $k=0$ and $k=2$ are excluded since complete graphs do not contradict our assumption of having a vertex not adjacent to
all other vertices. Hence, let $k \geq 3$.
\item Every two adjacent vertices in $\Gr{G}$ share a common neighbor, so $\Gr{G}$ contains a triangle.
Moreover, $\Gr{G}$ is $\CG{4}$-free since every two vertices have exactly one common neighbor,
so it must be also $\CoG{4}$-free. Hence, $\omega(\Gr{G}) = 3$.
\item We next show that $\chi(\Gr{G}) = 3$.
First, $\chi(\Gr{G}) \geq \omega(\Gr{G}) = 3$. We also need to show that $\chi(\Gr{G}) \leq 3$, which means that
three colors suffice to color all the vertices of $\Gr{G}$ in a way that no two adjacent vertices are assigned
the same color. This can be done recursively by observing that each edge belongs to exactly one triangle (otherwise,
the endpoints of that edge would have more than one common neighbor, which is not allowed). Moreover, each newly colored
vertex always completes a properly colored triangle, ensuring that the coloring remains valid at every step of the recursion
without requiring a fourth color.
\item By the sandwich theorem $\omega(\Gr{G}) \leq \vartheta(\CGr{G}) \leq \chi(\Gr{G})$, so $\omega(\Gr{G}) = \chi(\Gr{G}) = 3$
implies that $\vartheta(\CGr{G})=3$.
\item By Theorem~\ref{theorem: theta of srg}, where $\Gr{G}$ is a strongly
regular graph in the family $\srg{n}{k}{1}{1}$, we also get
\begin{align}
\label{eq1: 01.03.2025}
\vartheta(\CGr{G}) = 1 + \frac{k}{\sqrt{k-1}}.
\end{align}
This leads to a contradiction since, for all $k \geq 3$,
\begin{align}
& (k-2)^2 > 0, \nonumber \\
\Leftrightarrow \; & k^2 > 4(k-1), \nonumber \\
\Leftrightarrow \; & 1  + \frac{k}{\sqrt{k-1}} > 3,
\end{align}
which completes the proof of the theorem by contradiction. Note that every edge in $\Gr{G}$ lies on a triangle (by the theorem's assumption),
and $\Gr{G}$ consists of edge-disjoint triangles since it is $\CG{4}$-free.
\end{itemize}
\end{proof}

\subsection{Another proof of the friendship theorem relying on spectral properties of strongly regular graphs}
\label{subsection: alternative proof - spectrum of srg}

This subsection presents an alternative proof of the friendship theorem, relying on the adjacency spectrum
of strongly regular graphs. This second proof forms a variation of the proof provided in Chapter~44 of \cite{AignerZ18}.

From the point in Section~\ref{subsection: alternative proof - Lovasz function} where we get, by contradiction,
that $\Gr{G}$ is a strongly regular graph in the family $\srg{n}{k}{1}{1}$ with $k \geq 3$,
it is possible to get a contradiction in an alternative way, which relies on the
following known result for strongly regular graphs.

\begin{proposition}[Feasible parameters of strongly regular graphs]
\label{proposition: Feasible Parameters of Strongly Regular Graphs}
{\em Let $\Gr{G}$ be a strongly regular graph in the family $\srg{n}{d}{\lambda}{\mu}$. Then,
the following holds:
\begin{enumerate}[(1)]
\item $(n-d-1) \, \mu = d \, (d-\lambda-1)$.
\item $n-1 - \frac{2d+(n-1)(\lambda-\mu)}{\sqrt{(\lambda-\mu)^2+4(d-\mu)}}$ is a nonnegative even integer.
\item $6 | (n d \lambda)$.
\end{enumerate}}
\end{proposition}
\begin{proof}
This holds by relying on some basic properties of strongly regular graphs (see Chapter~21 of \cite{vanLintW02})
as follows.
\begin{itemize}
\item Item~1 is a combinatorial equality for strongly regular graphs. It is obtained by
letting the vertices in $\Gr{G}$ lie in three levels, where an arbitrarily selected vertex
is at the root in Level~0, its $d$ neighbors lie in Level~1, and all the others lie in Level~2.
The equality in Condition~1 then follows from a double counting argument
of the number of edges between Levels~1 and~2.
\item Item~2 holds by the integrality and nonnegativity of the multiplicities $m_{1,2}$ of the second-largest and least eigenvalues
of the adjacency matrix, respectively, which satisfy
\begin{align}
\label{eig-multiplicities-SRG}
0 \leq 2 m_{1,2} = n-1 \mp \frac{2d+(n-1)(\lambda-\mu)}{\sqrt{(\lambda-\mu)^2+4(d-\mu)}}.
\end{align}
\item Item~3 holds since the number of triangles in $\Gr{G}$ is given by $\tfrac16 nd \lambda$.
Indeed, every vertex has $d$ neighbors, every pair of adjacent vertices shares exactly $\lambda$ common neighbors
(forming a triangle), and each triangle is counted six times due to the six possible permutations of its three vertices.
\end{itemize}
\end{proof}
\begin{itemize}
\item By applying Item~1 in Proposition~\ref{proposition: Feasible Parameters of Strongly Regular Graphs}
to $\Gr{G}$ with the parameters $d=k$ and $\lambda = \mu = 1$, we get $n = k^2 - k+1$. This does not lead
yet to a contradiction since summing over all the degrees of the neighbors of an arbitrary vertex $u$ in
$\Gr{G}$ gives $k^2$.
Then, by the assumption of the theorem that every two vertices have exactly one common neighbor, it
follows that the above summation counts each vertex in $\Gr{G}$ exactly once, except for vertex
$u$ that is counted $k$ times. Hence, indeed $n = k^2-k+1$.
\item By Item~2 in Proposition~\ref{proposition: Feasible Parameters of Strongly Regular Graphs}
with $d=k$ and $\lambda = \mu = 1$, it follows that $\frac{k}{\sqrt{k-1}} \in \naturals$. Consequently,
$(k-1) | k^2$. Since $k^2 = (k-1)(k+1)+1$, it follows that $(k-1)|1$, so $k=2$.
If $k=2$, the only graph that satisfies the condition of Theorem~\ref{thm: friendship theorem}
is $\Gr{G} = \CoG{3}$, which also satisfies the assertion of the theorem. Hence, this argument
contradicts the assumption in the proof since it led to the conclusion that $\Gr{G}$ is a strongly
regular graph in the family $\srg{n}{k}{1}{1}$ with $k \geq 3$, which was subsequently rejected
(as $(k-1) \hspace*{-0.1cm} \not| \, 1$).
\end{itemize}

\subsection{Further discussions and insights on extensions of the friendship theorem}
\label{subsection:  extensions of the friendship theorem}

This subsection considers further generalizations and provides discussions on the friendship theorem and its proofs.
The following discussions are presented in Remarks~\ref{remark1: 01.03.25}--\ref{remark3: 01.03.25} and are supplemented
by examples that examine strongly regular graphs of different orders, both small and large. In particular,
Remark~\ref{remark3: 01.03.25} builds on our proof technique, which is based on the Lov\'{a}sz $\vartheta$-function of
strongly regular graphs, to offer insights into the extended friendship theorem studied in \cite{MertziosU16}. It also
establishes additional results, which are further illustrated through
Examples~\ref{example: cont. to remark 3}--\ref{example: complements of symplectic polar graphs}.

\begin{remark}
\label{remark1: 01.03.25}
{\em By the frienship theorem (Theorem~\ref{thm: friendship theorem}), every finite graph in which each
pair of distinct vertices shares exactly one common neighbor must be isomorphic to a windmill graph (see
Figure~\ref{fig:windmill_graph}). Consequently, its order $n$ must be an odd integer, and the size of
the graph is given by $\card{\E{\Gr{G}}} = \tfrac32 (n-1)$. By definition, a graph $\Gr{G}$
satisfying the assumption of Theorem~\ref{thm: friendship theorem} is $\CG{4}$-free.
A combinatorial proof, based on double counting, asserts that the size of a $\CG{4}$-free graph
$\Gr{G}$ on $n$ vertices satisfies (see pp.~200--201 in \cite{AignerZ18})
\begin{align}
\label{eq: UB - size of a C4-free graph}
\card{\E{\Gr{G}}} \leq \Bigl\lfloor \tfrac{1}{4} n (1+ \sqrt{4n-3}) \Bigr\rfloor.
\end{align}
Furthermore, it has been shown that this upper bound on the size of $\Gr{G}$ is nearly tight in
the sense that if $p$ is a prime number, and $n = p^2+p+1$, then a specific construction of a
$\CG{4}$-free graph on $n$ vertices has a size that is given by (see pp.~201--202 in \cite{AignerZ18})
\begin{align}
\label{eq: achievability - size of a C4-free graph}
\card{\E{\Gr{G}}} = \tfrac{1}{4} (n-1) (1+ \sqrt{4n-3}).
\end{align}
The significant difference between the size of a windmill graph on $n$ vertices, scaling linearly as $\tfrac32 n$,
and the largest size of a general $\CG{4}$-free graph on $n$ vertices, scaling super-linearly as $\tfrac12 n^{\frac32}$,
stems from additional structural constraint on $\Gr{G}$. Specifically, beyond being a $\CG{4}$-free graph, $\Gr{G}$ must
satisfy a stronger condition that uniquely determines it (up to isomorphism) as a windmill graph.}
\end{remark}

\begin{remark}
\label{remark2: 01.03.25}
{\em The friendship theorem was recently generalized in \cite{ChoiCK25} to directed graphs (digraphs),
allowing for asymmetry in liking relationships. The digraph version of the friendship theorem (see Theorem~1.1
of \cite{ChoiCK25}) states that if a finite, simple, directed graph has the property that every pair of
vertices has exactly one common out-neighbor, then it must be either:
\begin{enumerate}[(1)]
\item A $k$-regular digraph -- where each vertex has out-degree $k$ and in-degree $k$ -- on
$n=k^2-k+1$ vertices with $k \geq 2$, or
\item A so-called "fancy wheel digraph," which consists of a disjoint union of directed cycles
with an additional vertex that has arcs to and from every vertex on these cycles.
\end{enumerate}
Notably, the order $n$ of the $k$-diregular digraph matches that of the
undirected graph $\Gr{G}$ in the second proof (Section~\ref{subsection: alternative proof - spectrum of srg}),
in which it was established -- prior to reaching the contradiction -- that $\Gr{G}$ is a $k$-regular graph
of order $n=k^2-k+2$.}
\end{remark}

\begin{remark}
\label{remark3: 01.03.25}
{\em As a possible extension of the friendship theorem, Proposition~1 of \cite{MertziosU16} states that if $\Gr{G}$ is a
finite, simple, and undirected graph such that every pair of vertices has exactly $\ell$ common neighbors for some fixed $\ell \geq 2$,
then $\Gr{G}$ is a regular graph. Consequently, $\Gr{G}$ is a strongly regular graph in the family
$\srg{n}{d}{\ell}{\ell}$. Furthermore, by Item~1 of Proposition~\ref{proposition: Feasible Parameters of Strongly Regular Graphs}
(here), it follows by substituting $\lambda = \mu = \ell$ that the parameters $n, d, \ell$ are related by the equation
$(n-1) \ell = d(d-1)$. As an easy generalization of \eqref{eq1: 01.03.2025} in our proof of the friendship theorem,
it follows from Theorem~\ref{theorem: theta of srg} that
\begin{align}
\label{eq2: 01.03.2025}
\vartheta(\CGr{G}) = 1 + \frac{d}{\sqrt{d-\ell}}.
\end{align}
Similar to our proof of the friendship theorem, which is based on the sandwich theorem stating that
$\omega(\Gr{G}) \leq \vartheta(\CGr{G}) \leq \chi(\Gr{G})$, it follows that the clique and chromatic
numbers of $\Gr{G}$ satisfy, respectively,
\begin{align}
\label{eq3: 01.03.2025}
\omega(\Gr{G}) \leq 1 + \biggl\lfloor \frac{d}{\sqrt{d-\ell}} \biggr\rfloor, \\
\label{eq4: 01.03.2025}
\chi(\Gr{G}) \geq 1 + \biggl\lceil \frac{d}{\sqrt{d-\ell}} \biggr\rceil.
\end{align}
Furthermore, by Corollary~\ref{corollary: property - SRGs or v.t.}, it follows from \eqref{eq2: 01.03.2025} that
\begin{align}
\label{eq5: 01.03.2025}
\vartheta(\Gr{G}) = \frac{n \sqrt{d-\ell}}{d + \sqrt{d-\ell}}.
\end{align}
Similarly, by the sandwich theorem, the independence number of $\Gr{G}$ and the chromatic number of its complement $\CGr{G}$ satisfy
the following bounds:
\begin{align}
\label{eq6: 01.03.2025}
\indnum{\Gr{G}} \leq \biggl\lfloor \frac{n \sqrt{d-\ell}}{d + \sqrt{d-\ell}} \biggr\rfloor, \\
\label{eq7: 01.03.2025}
\chi(\CGr{G}) \geq \biggl\lceil \frac{n \sqrt{d-\ell}}{d + \sqrt{d-\ell}} \biggr\rceil.
\end{align}
Hence, our proposed proof of the friendship theorem that relies on the Lov\'{a}sz $\vartheta$-function of strongly regular graphs,
combined with Theorem~\ref{theorem: theta of srg} and Corollary~\ref{corollary: property - SRGs or v.t.} (here), and Proposition~1 of
\cite{MertziosU16}, lead to bounds on graph invariants of $\Gr{G}$ and $\CGr{G}$ with respect to the extended version of the friendship
theorem of \cite{MertziosU16} for an arbitrary $\ell \geq 2$.}
\end{remark}

The next three examples, preceded by definitions and notation, refer to Remark~\ref{remark3: 01.03.25}.
\begin{definition}[Cartesian product of graphs]
\label{def:cartesian product of graphs}
{\em Let $\Gr{G}$ and $\Gr{H}$ be graphs.
The Cartesian product, denoted by $\Gr{G} \, \square \, \Gr{H}$, is a graph with a vertex set
that is given by $\V{\Gr{G} \, \square \, \Gr{H}} = \V{\Gr{G}} \times \V{\Gr{H}}$
(i.e., it is equal to the Cartesian product of the vertex sets of $\Gr{G}$ and $\Gr{H}$),
and its edge set is characterized by the property that distinct vertices $(g, h)$ and
$(g', h')$, where $g, g' \in \V{\Gr{G}}$ and $h, h' \in \V{\Gr{H}}$, are adjacent in
$\Gr{G} \, \square \, \Gr{H}$ if and only if one of the following conditions holds: \newline
(1) $g = g'$ and $\{h, h'\} \in \E{\Gr{H}}$, or \; (2) $\{g, g'\} \in \E{\Gr{G}}$ and $h = h'$.}
\end{definition}

\begin{definition}[Line Graph]
\label{definition: line graphs}
{\em Let $\Gr{G} = (\Vertex, \Edge)$ be a graph. The line graph of $\Gr{G}$, denoted by $\mathrm{L}(\Gr{G})$,
is a graph whose vertices are the edges of $\Gr{G}$, and two vertices are adjacent in the line graph
$\mathrm{L}(\Gr{G})$ if the corresponding edges are incident in $\Gr{G}$.}
\end{definition}

\begin{example}
\label{example: cont. to remark 3}
{\em The Shrikhande graph and the Cartesian product of two complete graphs on~4 vertices, $\CoG{4} \square \CoG{4}$,
are the two nonisomorphic strongly regular graphs in the family $\srg{16}{6}{2}{2}$, where each pair of distinct vertices
has exactly two common neighbors. Note that the parameters $n=16$, $d=6$, and $\ell=2$ indeed satisfy the equality
$(n-1)\ell=d(d-1)$. Let $\Gr{G}_1$ and $\Gr{G}_2$ denote, respectively, these two graphs. By \eqref{eq3: 01.03.2025},
\eqref{eq4: 01.03.2025}, \eqref{eq6: 01.03.2025}, and \eqref{eq7: 01.03.2025}, we get
\begin{align}
\label{eq8: 01.03.2025}
\indnum{\Gr{G}} \leq 4, \quad \omega(\Gr{G}) \leq 4, \quad \chi(\Gr{G}) \geq 4, \quad \chi(\CGr{G}) \geq 4,
\end{align}
which hold for $\Gr{G}_1$ and $\Gr{G}_2$. A comparison with their exact values (computed by \cite{SageMath}) gives
\begin{align}
\label{eq9: 01.03.2025}
& \indnum{\Gr{G}_1} = 4, \quad \omega(\Gr{G}_1) = 3, \quad \chi(\Gr{G}_1) = 4, \quad \chi(\CGr{G}_1) = 6, \\
\label{eq10: 01.03.2025}
& \indnum{\Gr{G}_2} = 4, \quad \omega(\Gr{G}_2) = 4, \quad \chi(\Gr{G}_2) = 4, \quad \chi(\CGr{G}_2) = 4,
\end{align}
so the four bounds on the graph invariants in \eqref{eq8: 01.03.2025} are tight for $\Gr{G} = \Gr{G}_2$.}
\end{example}

\begin{example}
\label{example 2: cont. to remark 3}
{\em For $n \geq 4$, the line graph of the complete graph on $n$ vertices, $\Gr{G} = \mathrm{L}(\CoG{n})$,
is a strongly regular graph in the family $\srg{\tfrac12 n(n-1)}{2(n-2)}{n-2}{4}$. Consequently, the line graph
$\mathrm{L}(\CoG{6})$ is in the family $\srg{15}{8}{4}{4}$, so every pair of distinct vertices in that graph has exactly
four common neighbors. Using the SageMath software \cite{SageMath} gives $\indnum{\Gr{G}} = 3$, $\omega(\Gr{G}) = 5$, $\chi(\Gr{G})=5$,
and $\chi(\CGr{G}) = 4$. Substituting the graph parameters $n=15$, $d=8$, and $\ell=4$ into \eqref{eq3: 01.03.2025}--\eqref{eq7: 01.03.2025}
shows that the three bounds in \eqref{eq3: 01.03.2025}, \eqref{eq4: 01.03.2025}, and \eqref{eq6: 01.03.2025} are tight. However,
the bound in \eqref{eq7: 01.03.2025} is not tight as it gives $\chi(\CGr{G}) \geq 3$, whereas its exact value is~4.}
\end{example}

\begin{example}[Complements of symplectic polar graphs]
\label{example: complements of symplectic polar graphs}
{\em The symplectic polar graphs are strongly regular graphs, belonging to the families
$\srg{v}{k}{\lambda}{\mu}$, where
\begin{align}
\label{eq: symplectic}
v = \frac{q^{2n}-1}{q-1}, \quad k = \frac{q(q^{2n-2}-1)}{q-1}, \quad
\lambda = \frac{q^{2n-2}-1}{q-1} - 2, \quad
\mu = \frac{q^{2n-2}-1}{q-1},
\end{align}
for all $n \in \naturals$ and $q \geq 2$ that is a prime power. Consequently,
$\lambda = \mu-2$ and $\mu = \frac{k}{q}$. That symplectic polar graph
is denoted by $\Gr{G} = \Sp(2n,q)$ (see Section~2.5 in \cite{BrouwerM22}).
The complement graph of a strongly regular graph in the family $\srg{\nu}{k}{\lambda}{\mu}$
is a strongly regular graph in the family $\srg{v}{v-k-1}{v-2k+\mu-2}{v-2k+\lambda}$ (see
Proposition~\ref{proposition: srg complement}).
Substituting $v, k, \lambda, \mu$ in \eqref{eq: symplectic}
gives that $\Gr{H} = \CGr{G}$ is a strongly regular graph in the family
\begin{align}
\label{eq: H complement symplectic graph}
\vspace*{0.1cm}
\SRG \, \Biggl(\frac{q^{2n}-1}{q-1}, q^{2n-1}, q^{2n-2} (q-1), q^{2n-2} (q-1)\Biggr).
\end{align}
Consequently, the family of the complements of the symplectic polar graphs,
$\Gr{H} = \overline{\Sp(2n,q)}$, where $n \in \naturals$ and $q \geq 2$ is a prime power,
forms an infinite family of strongly regular graphs which are characterized by the property
that every pair of vertices in such a graph has an identical number of common neighbors,
whose value is given by $\ell = q^{2n-2} (q-1)$.

\begin{table}[hbt]
\centering
\caption{(Example~\ref{example: complements of symplectic polar graphs})
Complements of symplectic polar graphs, the families of strongly regular graphs to
which they belong (see \eqref{eq: H complement symplectic graph}), the fixed number ($\ell$)
of common neighbors for each pair of distinct vertices, the values of the Lov\'{a}sz
$\vartheta$-function for these graphs and their complements, and their exact
matchings with the independence and chromatic numbers of these graphs.}
\label{Table: complements of symplectic polar graphs}
\vspace*{0.1cm}
\renewcommand{\arraystretch}{1.5}
\centering
\begin{tabular}{cclrrrrr}
\hline
$n$ & $q$ & $\hspace*{0.7cm} \Gr{H} = \overline{\Sp(2n,q)}$ & $\ell$ & $\vartheta(\Gr{H})$ & $\indnum{\Gr{H}}$ & $\vartheta(\CGr{H})$ & $\chi(\Gr{H})$ \\[0.1cm] \hline
$3$ & $2$ & $\srg{63}{32}{16}{16}$ & $16$ & $7$ & $7$ & $9$ & $9$\\
$3$ & $3$ & $\srg{364}{243}{162}{162}$ & $162$ & $13$ & $13$ & $28$ & $28$ \\
$3$ & $4$ & $\srg{1365}{1024}{768}{768}$ & $768$ & $21$ & $21$ & $65$ & $65$ \\
$3$ & $5$ & $\srg{3906}{3125}{2500}{2500}$ & $2500$ & $31$ & $31$ & $126$ & $126$ \\
$3$ & $7$ & $\srg{19608}{16807}{14406}{14406}$ & $14406$ & $57$ & $57$ & $344$ & $344$ \\
$4$ & $2$ & $\srg{255}{128}{64}{64}$ & $64$ & $15$ & $15$ & $17$ & $17$ \\
$4$ & $3$ & $\srg{3280}{2187}{1458}{1458}$ &$1458$ & $40$ & $40$ & $82$ & $82$ \\
$4$ & $4$ & $\srg{21845}{16384}{12288}{12288} $ & $12288$ & $85$ & $85$ & $257$ & $257$ \\
$5$ & $2$ & $\srg{1023}{512}{256}{256}$ & $256$ & $31$ & $31$ & $33$ & $33$ \\
$5$ & $3$ & $\srg{29524}{19683}{13122}{13122}$ & $13122$ & $121$ & $121$ & $244$ & $244$ \\
$5$ & $4$ & $\srg{349525}{262144}{196608}{196608}$ & $196608$ & $341$ & $341$ & $1025$ & $1025$ \\
$6$ & $2$ & $\srg{4095}{2048}{1024}{1024}$ & $1024$ & $63$ & $63$ & $65$ & $65$ \\
$6$ & $3$ & $\srg{265720}{177147}{118098}{118098}$ & $118098$ & $364$ & $364$ & $730$ & $730$ \\
$7$ & $2$ & $\srg{16383}{8192}{4096}{4096}$ & $4096$ & $127$ & $127$ & $129$ & $129$ \\
$8$ & $2$ & $\srg{65535}{32768}{16384}{16384}$ & $16384$ & $255$ & $255$ & $257$ & $257$ \\
$9$ & $2$ & $\srg{262143}{131072}{65536}{65536}$ & $65536$ & $511$ & $511$ & $513$ & $513$
\\[0.1cm] \hline
\end{tabular}
\end{table}
In regard to their graph invariants (as illustrated in Table~\ref{Table: complements of symplectic polar graphs}),
by Section 2.5.4 of \cite{BrouwerM22},
\begin{align}
\label{eq1: 2.3.25}
\indnum{\Gr{H}} = \omega(\Gr{G}) = \frac{q^n-1}{q-1},
\end{align}
and, by Theorem~3.29 of \cite{Sason24},
\begin{align}
\label{eq2: 2.3.25}
& \vartheta(\Gr{H}) = \indnum{\Gr{H}} = \frac{q^n-1}{q-1}, \\
\label{eq3: 2.3.25}
& \vartheta(\CGr{H}) = \chi(\Gr{H}) = q^n+1.
\end{align}
Consequently, it follows that the bounds in \eqref{eq6: 01.03.2025} and \eqref{eq7: 01.03.2025},
applied to $\Gr{H}$ for its independence and chromatic numbers, are tight. Numerical results are
presented in Table~\ref{Table: complements of symplectic polar graphs}.}
\end{example}

\section{Spanning Subgraphs of Strongly Regular Graphs}
\label{section: spanning subgraphs}

A spanning subgraph is obtained by removing some edges from the original graph while preserving
all its vertices. Let $\Gr{G}$ and $\Gr{H}$ be graphs of the same order, belonging to the
strongly regular graph families $\srg{n}{d}{\lambda}{\mu}$ and $\srg{n}{d'}{\lambda'}{\mu'}$, respectively.
The specific structure of these graphs may not necessarily be specified, as multiple nonisomorphic strongly
regular graphs can exist within a given family \cite{Brouwer,BrouwerM22}. We next derive a
necessary condition for one of these graphs to be a spanning subgraph of the other, relying on the
Lov\'{a}sz $\vartheta$-functions of these graphs and their complements.

We rely on the following simple lemma, which relates the Lov\'{a}sz $\vartheta$-functions
of two graphs (not necessarily strongly regular), where one is a spanning subgraph of the other.
\begin{lemma}
\label{lemma: spanning subgraph}
{\em Let $\Gr{G}$ and $\Gr{H}$ be undirected and simple graphs with an identical vertex set.
If $\Gr{H}$ is a spanning subgraph of $\Gr{G}$, then
\begin{align}
& \vartheta(\CGr{G}) \geq \vartheta(\CGr{H}),   \label{eq1: spanning subgraph} \\
& \vartheta(\Gr{G}) \leq \vartheta(\Gr{H}).   \label{eq2: spanning subgraph}
\end{align}}
\end{lemma}
\begin{proof}
Inequality \eqref{eq1: spanning subgraph} follows from \eqref{eq: Lovasz theta function} since every
orthonormal representation of $\CGr{G}$ is also an orthonormal representation of $\CGr{H}$. This holds
since if $\{i,j\} \not\in \E{\CGr{H}}$ and $i \neq j$, then $\{i,j\} \in \E{\Gr{H}} \subseteq \E{\Gr{G}}$,
and therefore $\{i,j\} \not\in \E{\CGr{G}}$; hence, if ${\bf{u}}_i^{\mathrm{T}} {\bf{u}}_j = 0$ for all
$\{i,j\} \not\in \E{\CGr{G}}$, then the same also holds for all $\{i,j\} \not\in \E{\CGr{H}}$.
Inequality \eqref{eq1: spanning subgraph} also follows alternatively from the SDP problem in \eqref{eq: SDP}
since every feasible solution that corresponds to $\vartheta(\CGr{H})$ is also a feasible solution that
corresponds to $\vartheta(\CGr{G})$. Inequality \eqref{eq2: spanning subgraph} follows from \eqref{eq1: spanning subgraph}
since, by definition, $\Gr{H}$ is a spanning subgraph of $\Gr{G}$ if and only if $\CGr{G}$ is a spanning subgraph of $\CGr{H}$.
\end{proof}

\begin{remark}
\label{remark: on the two inequalities for a spanning subgraph}
{\em Let $\Gr{G}$ and $\Gr{H}$ be two undirected and simple graphs on $n$ vertices, and suppose that we do not know yet
whether $\Gr{H}$ is a spanning subgraph of $\Gr{G}$.
In light of Corollary~\ref{corollary: property - SRGs or v.t.}, it follows that if $\Gr{G}$ and $\Gr{H}$ are
either strongly regular or vertex-transitive graphs, then each of the two inequalities \eqref{eq1: spanning subgraph}
and \eqref{eq2: spanning subgraph} hold if and only if the other inequality holds.
In general, by Corollary~2 of \cite{Lovasz79_IT}, we have $\vartheta(\Gr{G}) \, \vartheta(\CGr{G}) \geq n$ and
$\vartheta(\Gr{H}) \, \vartheta(\CGr{H}) \geq n$, so, if $\Gr{H}$ is not necessarily a spanning subgraph of $\Gr{G}$,
then the satisfiability of one of the inequalities in \eqref{eq1: spanning subgraph} and \eqref{eq2: spanning subgraph}
does not necessarily imply the satisfiability of the other inequality. }
\end{remark}

Before proceeding with the application of Lemma~\ref{lemma: spanning subgraph} to strongly regular graphs,
let us consider the problem by only relying on Definition~\ref{definition: srg} of strongly regular graphs.
Let $\Gr{G}$ and $\Gr{H}$ be strongly regular graphs on $n$ vertices, belonging to the families
$\srg{n}{d}{\lambda}{\mu}$ and $\srg{n}{d'}{\lambda'}{\mu'}$, respectively. If $\Gr{H}$ is
a spanning subgraph of $\Gr{G}$, then the following three inequalities hold:
\begin{align}
\label{eq: three simple inequalities}
d > d', \quad \lambda \geq \lambda', \quad \min\{\lambda, \mu\} \geq \mu'.
\end{align}
This follows from the fact that deleting edges can only decrease the fixed degree of the vertices in a
regular graph. Moreover, since every pair of adjacent vertices in $\Gr{H}$ are also adjacent in $\Gr{G}$,
the fixed number of common neighbors of any two adjacent vertices in $\Gr{H}$ must be smaller than or equal to
the fixed number of common neighbors of the corresponding vertices in $\Gr{G}$. Furthermore, the fixed number
of common neighbors of any two nonadjacent vertices in $\Gr{H}$, which are either adjacent or nonadjacent
vertices in $\Gr{G}$, cannot be larger than the minimum between the fixed numbers of common neighbors of
adjacent or nonadjacent vertices in $\Gr{G}$.

We obtain a further necessary condition, which specifies Lemma~\ref{lemma: spanning subgraph} for strongly
regular graphs.

\begin{corollary}
\label{corollary: necessary cond.}
{\em Let $\Gr{G}$ and $\Gr{H}$ be strongly regular graphs on $n$ vertices, belonging to the families
$\srg{n}{d}{\lambda}{\mu}$ and $\srg{n}{d'}{\lambda'}{\mu'}$, respectively. Then, a necessary
condition for $\Gr{H}$ to be a spanning subgraph of $\Gr{G}$ is that their parameters satisfy the
inequality
\begin{align}
\label{eq: necessary cond.}
\frac{d}{d'} \; \frac{\mu'-\lambda' + \sqrt{(\mu'-\lambda')^2 + 4(d'-\mu')}}{\mu-\lambda + \sqrt{(\mu-\lambda)^2 + 4(d-\mu)}} \geq 1.
\end{align}}
\end{corollary}
\begin{proof}
The result is obtained by combining Theorem~\ref{theorem: theta of srg} (see \eqref{eq2: 15.02.25}, \eqref{eq3: 15.02.25})
and Lemma~\ref{lemma: spanning subgraph} (see \eqref{eq1: spanning subgraph}), followed by straightforward algebra.
\end{proof}

\begin{example}
{\em Let $\Gr{G}$ and $\Gr{H}$ be strongly regular graphs belonging to the families $\srg{45}{28}{15}{21}$ and
$\srg{45}{22}{10}{11}$, respectively (by \cite{Brouwer}, there exist strongly regular graphs with these parameters).
The inequalities in \eqref{eq: three simple inequalities} are satisfied, so the question if $\Gr{H}$
can be a spanning subgraph of $\Gr{G}$ is left open according to these inequalities. However,
by Corollary~\ref{corollary: necessary cond.}, $\Gr{H}$ (or any graph isomorphic to $\Gr{H}$) cannot be a spanning
subgraph of $\Gr{G}$ since the left-hand side of \eqref{eq: necessary cond.} is equal to
$\frac{3 \sqrt{5}+1}{11} = 0.7007 \ldots < 1$, so the necessary condition in \eqref{eq: necessary cond.} is violated.}
\end{example}

The following result extends Corollary~\ref{corollary: necessary cond.} to a more general setting, allowing for
regular spanning subgraphs of a strongly regular graph and, even more broadly, regular spanning subgraphs of regular
graphs. This extension relaxes the strong regularity requirement in Corollary~\ref{corollary: necessary cond.},
replacing it with the milder condition of regularity. This generalization comes, however, at a certain cost:
it necessitates knowledge of the second-largest and least eigenvalues of their adjacency spectra, whereas
Corollary~\ref{corollary: necessary cond.} requires to only know the parameter vectors that characterize the
strongly regular graphs $\Gr{G}$ and $\Gr{H}$.
This distinction arises because nonisomorphic strongly regular graphs in the same family
$\srg{n}{d}{\lambda}{\mu}$, for any feasible parameters $(n,d,\lambda,\mu)$, are cospectral, with their adjacency
spectra uniquely determined by these parameters. In contrast, nonisomorphic $d$-regular graphs on $n$ vertices
are not necessarily cospectral.

\begin{corollary}
\label{corollary: 26.02.2025}
{\em Let $\Gr{G}$ and $\Gr{H}$ be noncomplete and nonempty $d$-regular and $d'$-regular graphs on $n$ vertices,
with $d' \leq d$. Then, $\Gr{H}$ is a spanning subgraph of $\Gr{G}$ if the following inequality is satisfied:
\begin{align}
\label{eq1: 26.02.2025}
& \frac{n \bigl(1+\Eigval{2}{\Gr{G}}\bigr)}{n-d+\Eigval{2}{\Gr{G}}} + \frac{d'}{\Eigval{n}{\Gr{H}}} \geq 1,
\end{align}
where $\Eigval{2}{\Gr{G}}$ is the second-largest eigenvalue of the adjacency matrix of $\Gr{G}$, and
$\Eigval{n}{\Gr{H}}$ is the smallest eigenvalue of the adjacency matrix of $\Gr{H}$.}
\end{corollary}
\begin{proof}
The necessary condition in \eqref{eq1: 26.02.2025} holds by combining inequality \eqref{eq2: spanning subgraph}
with the lower bound on $\vartheta(\Gr{G})$ and the upper bound on $\vartheta(\Gr{H})$ in the leftmost and rightmost
inequalities in \eqref{eq:21.10.22a1}, respectively.
\end{proof}
\begin{remark}
{\em Combining inequality \eqref{eq1: spanning subgraph}
with the lower bound on $\vartheta(\CGr{H})$ and the upper bound on $\vartheta(\CGr{G})$ in the leftmost and rightmost
inequalities in \eqref{eq:21.10.22a2}, respectively, gives
\begin{align}
\label{eq2: 26.02.2025}
\frac{n-d+\Eigval{2}{\Gr{G}}}{1+\Eigval{2}{\Gr{G}}} \leq -\frac{n \Eigval{n}{\Gr{H}}}{d' - \Eigval{n}{\Gr{H}}},
\end{align}
which is equivalent to \eqref{eq1: 26.02.2025} (note that for a noncomplete regular graph $\Gr{G}$, its second-largest
eigenvalue satisfies $\Eigval{2}{\Gr{G}} > -1$, and $\Eigval{n}{\Gr{H}}<0$ for a nonempty graph $\Gr{H}$).
Corollary~\ref{corollary: 26.02.2025} consequently contains a single inequality since the two inequalities in
Lemma~\ref{lemma: spanning subgraph}, combined with \eqref{eq:21.10.22a1} and \eqref{eq:21.10.22a2}, yield an
identical inequality.}
\end{remark}

\section{Induced Subgraphs of Strongly Regular Graphs}
\label{section: induced subgraphs}

An induced subgraph is obtained by deleting vertices from the original graph, along with all their incident edges.
In analogy to Section~\ref{section: spanning subgraphs}, which is focused on spanning subgraphs of strongly
regular graphs, and an extension of that analysis to regular graphs and spanning subgraphs, this section
examines induced subgraphs in a similar manner.
Specifically, we first derive necessary conditions for a graph to be an induced subgraph of a strongly regular graph
in the family $\srg{n}{d}{\lambda}{\mu}$. This analysis relies not only on the Lov\'{a}sz $\vartheta$-function for
strongly regular graphs but also on their energy.

The graph energy is a graph invariant originally introduced in \cite{Gutman78}, exploring its application in
chemistry. A comprehensive treatment of graph energy can be found in \cite{LiSG12}.
\begin{definition}[Energy of a graph, \cite{LiSG12}]
\label{definition: graph energy}
{\em The energy of a graph $\Gr{G}$ on $n$ vertices, denoted by $\Enr{\Gr{G}}$, is given by
\begin{align}
\label{eq: graph energy}
\Enr{\Gr{G}} \eqdef \sum_{i=1}^n |\lambda_i(\Gr{G})|.
\end{align}}
\end{definition}

The following lemma relates the energy and the Lov\'{a}sz $\vartheta$-function of an induced subgraph to
those of the original graph.
It is general, and it is then applied to strongly regular graphs, and extended to regular graphs.
It partially relies on the Cauchy interlacing theorem, which is cited as follows:

\begin{theorem}[Cauchy interlacing theorem, \cite{BrouwerH2011}]
\label{thm:interlacing}
{\em Let $\lambda_{1} \ge \ldots \ge \lambda_{n}$ be the eigenvalues of a symmetric matrix $\mathbf{M}$
and let $\mu_{1}\ge\ldots\ge\mu_{m}$ be the eigenvalues of a principal $m \times m$ submatrix
of $\mathbf{M}$ (i.e., a submatrix that is obtained by deleting the same set of rows and columns
from $M$). Then, $\lambda_{i}\ge\mu_{i}\ge\lambda_{n-m+i}$ for $i=1,\ldots,m$.}
\end{theorem}

\begin{lemma}
\label{lemma: induced subgraph}
{\em Let $\Gr{G}$ and $\Gr{H}$ be finite, undirected and simple graphs.
If $\Gr{H}$ is an induced subgraph of $\Gr{G}$, then
\begin{align}
\label{eq: induced subgraph - Lovasz}
& \vartheta(\Gr{H}) \leq \vartheta(\Gr{G}),  \\
\label{eq: induced subgraph - energy}
& \Enr{\Gr{H}} \leq \Enr{\Gr{G}}.
\end{align}}
\end{lemma}
\begin{proof}
Inequality~\eqref{eq: induced subgraph - Lovasz} holds since the vertex set of an induced subgraph $\Gr{H}$ is a subset of the vertex set
of the original graph $\Gr{G}$, and the adjacency and nonadjacency relations of the unremoved vertices in $\Gr{H}$ are preserved as compared
to those in $\Gr{G}$. The result then follows from \eqref{eq: Lovasz theta function} since every orthonormal representation of $\Gr{G}$
yields an orthonormal representation of $\Gr{H}$ by taking the subset of the orthonormal vectors $\{{\bf{u}}_i\}$ that correspond to the
remaining vertices in $\Gr{H}$.

Inequality~\eqref{eq: induced subgraph - energy} follows from the Cauchy interlacing theorem (Theorem~\ref{thm:interlacing}),
applied to the adjacency matrix of the graph $\Gr{G}$.
Note that the adjacency matrix of the induced subgraph $\Gr{H}$ is obtained from the adjacency matrix of $\Gr{G}$
by deleting the rows and columns that correspond to the deleted vertices from $\Gr{G}$.
\end{proof}

\begin{remark}
{\em Inequality~\eqref{eq: induced subgraph - energy} states that the energy of an induced subgraph is at most the
energy of the original graph. However, this result does not generally extend to spanning subgraphs (see p.~64 of \cite{LiSG12}).
A simple counterexample demonstrating that the energy of a spanning subgraph can exceed that of the original graph
is given by considering $\Gr{G}=\CG{4}$ and removing an edge to obtain the path graph $\Gr{H} = \PathG{4}$ (on 4~vertices)
as a spanning subgraph of $\Gr{G}$.
The adjacency spectra of $\Gr{G}$ and $\Gr{H}$ are, respectively, given by $\{-2,0, 0, 2\}$ and
$\bigl\{\tfrac12 (\sqrt{5}+1), \, \tfrac12 (\sqrt{5}-1), \, -\tfrac12 (\sqrt{5}-1), -\tfrac12 (\sqrt{5}+1)\bigr\}$,  so it follows
that $\Enr{\Gr{G}} = 4 < 2 \sqrt{5} = \Enr{\Gr{H}}$.
On the other hand, a sufficient condition where the energy of a spanning subgraph does not exceed that of the original graph
is given in Theorem~4.20 of \cite{LiSG12}. This highlights that the energy of a spanning subgraph can be either greater or
smaller than that of the original graph, which is why such a condition is not included in Section~\ref{section: spanning subgraphs}.}
\end{remark}

We next utilize the spectral properties of strongly regular graphs, listed in Theorem~\ref{theorem: eigenvalues of srg},
to derive a closed-form expression for their energy.
\begin{theorem}[The adjacency spectra of strongly regular graphs, Chapter~21 of \cite{vanLintW02}]
\label{theorem: eigenvalues of srg}
{\em The following spectral properties of strongly regular graphs hold:
\begin{enumerate}[(1)]
\item \label{Item 1: eigenvalues of srg}
A strongly regular graph has at most three distinct eigenvalues.
\item \label{Item 2: eigenvalues of srg}
Let $\Gr{G}$ be a connected strongly regular graph in the family $\SRG(n,d,\lambda,\mu)$ (i.e., $\mu>0$).
Then, its adjacency spectrum consists of three distinct eigenvalues, where the largest
eigenvalue is given by $\Eigval{1}{\Gr{G}} = d$ with multiplicity~1, and the other two
distinct eigenvalues of its adjacency matrix are given by
\begin{align}
\label{eigs-SRG}
p_{1,2} = \tfrac12 \, \biggl( \lambda - \mu \pm \sqrt{ (\lambda-\mu)^2 + 4(d-\mu) } \, \biggr),
\end{align}
with the respective multiplicities
\begin{align}
\label{eq2:eig-multiplicities-SRG}
m_{1,2} = \tfrac12 \, \Biggl( n-1 \mp \frac{2d+(n-1)(\lambda-\mu)}{\sqrt{(\lambda-\mu)^2+4(d-\mu)}} \, \Biggr).
\end{align}
\item \label{Item 3: eigenvalues of srg}
A connected regular graph with exactly three distinct eigenvalues is strongly regular.
\item \label{Item 4: eigenvalues of srg}
Connected strongly regular graphs, for which $2d+(n-1)(\lambda-\mu) \neq 0$, have integral eigenvalues and their respective
multiplicities are distinct.
\item \label{Item 5: eigenvalues of srg}
A connected regular graph is strongly regular if and only if it has three distinct eigenvalues, where the largest
eigenvalue is of multiplicity~1.
\item \label{Item 6: eigenvalues of srg}
A disconnected strongly regular graph is a disjoint union of $m$ identical complete graphs $\CoG{r}$, where $m \geq 2$ and $r \in \naturals$.
It belongs to the family $\srg{mr}{r-1}{r-2}{0}$, and its adjacency spectrum is $\{ (r-1)^{[m]}, (-1)^{[m(r-1)]} \}$, where superscripts
indicate the multiplicities of the eigenvalues, thus having two distinct eigenvalues.
\end{enumerate}}
\end{theorem}

A closed-form expression for the energy of strongly regular graphs as a function of a general
parameter vector $(n,d,\lambda,\mu)$ does not appear to be explicitly available in the literature (see, e.g.,
\cite{BrouwerM22, Haemers08, KoolenM01, LiSG12}). Such an expression is next introduced for further analysis.
\begin{lemma}[The energy of strongly regular graphs]
\label{lemma: energy of strongly regular graphs}
{\em The energy of a strongly regular graph $\Gr{G}$ in the family $\srg{n}{d}{\lambda}{\mu}$ is given by
\begin{align}
\label{eq: energy of strongly regular graphs}
\Enr{\Gr{G}} = d \, \Biggl( 1 + \frac{2(n-d)+\lambda-\mu}{\sqrt{(\lambda-\mu)^2 + 4(d-\mu)}} \Biggr).
\end{align}}
\end{lemma}
\begin{proof}
A connected strongly regular graph $\Gr{G}$ has an adjacency spectrum consisting of exactly three distinct eigenvalues, as
determined in Item~2 of Theorem~\ref{theorem: eigenvalues of srg}.
The largest eigenvalue of $\Gr{G}$ is~$d$ with multiplicity~1, the second largest eigenvalue is $\lambda_2(\Gr{G}) = p_1$ with
multiplicity~$m_1$, and the smallest eigenvalue is $\lambda_{\min}(\Gr{G}) = p_2$ with multiplicity~$m_2$ (see
\eqref{eigs-SRG} and \eqref{eq2:eig-multiplicities-SRG}). Consequently, by \eqref{eq: graph energy}, the energy of $\Gr{G}$ is given by
\begin{align}
\label{eq1: 03.03.2025}
\Enr{\Gr{G}} = d + m_1 |p_1| + m_2 |p_2|.
\end{align}
Substituting \eqref{eigs-SRG} and \eqref{eq2:eig-multiplicities-SRG} into \eqref{eq1: 03.03.2025} gives
\begin{align}
\Enr{\Gr{G}} &= d + \tfrac12 \, \Biggl( \lambda - \mu + \sqrt{ (\lambda-\mu)^2 + 4(d-\mu) } \, \Biggr) \cdot \tfrac12 \, \Biggl( n-1 - \frac{2d+(n-1)(\lambda-\mu)}{\sqrt{(\lambda-\mu)^2+4(d-\mu)}} \, \Biggr) \nonumber \\
& \hspace*{0.2cm} + \tfrac12 \, \Biggl( \mu - \lambda + \sqrt{ (\lambda-\mu)^2 + 4(d-\mu) } \Biggr) \cdot \tfrac12 \,
\Biggl( n-1 + \frac{2d+(n-1)(\lambda-\mu)}{\sqrt{(\lambda-\mu)^2+4(d-\mu)}} \, \Biggr) \nonumber  \\[0.1cm]
& = d + \tfrac12 (n-1) \sqrt{ (\lambda-\mu)^2 + 4(d-\mu) } + \frac{2d+(n-1)(\lambda-\mu)}{\sqrt{(\lambda-\mu)^2+4(d-\mu)}}
\cdot \tfrac12 (\mu-\lambda) \nonumber \\[0.1cm]
&= d + \frac{2(n-1)(d-\mu)+d(\mu-\lambda)}{\sqrt{(\lambda-\mu)^2+4(d-\mu)}} \nonumber \\[0.1cm]
&= d + \frac{\bigl[2(n-1)+\mu-\lambda \bigr] \, d - 2(n-1)\mu}{\sqrt{(\lambda-\mu)^2+4(d-\mu)}} \nonumber \\[0.1cm]
&=  d \, \Biggl( 1 + \frac{2(n-d)+\lambda-\mu}{\sqrt{(\lambda-\mu)^2 + 4(d-\mu)}} \Biggr),  \label{eq: 18.02.25}
\end{align}
where the last equality in \eqref{eq: 18.02.25} can be verified by relying on the
connection between the four parameters of a strongly regular graph as given in Item~1 of
Proposition~\ref{proposition: Feasible Parameters of Strongly Regular Graphs}, followed
by straightforward algebra. This establishes \eqref{eq: energy of strongly regular graphs}
for all connected strongly regular graphs.

If $\Gr{G}$ is a disconnected strongly regular graph then, by Item~6 of Theorem~\ref{theorem: eigenvalues of srg},
it is a disjoint union of $m$ identical complete graphs $\CoG{r}$, where $m \geq 2$ and $r \in \naturals$. The
adjacency spectrum of $\Gr{G}$ is then given by $\{ (r-1)^{[m]}, (-1)^{[m(r-1)]} \}$, so the graph energy is
$\Enr{\Gr{G}} = 2m(r-1)$. That strongly regular graph $\Gr{G}$ belongs to the family $\srg{mr}{r-1}{r-2}{0}$,
and substituting the parameters $n=mr$, $d=r-1$, $\lambda = r-2$, and $\mu=0$ into \eqref{eq: energy of strongly regular graphs}
verifies that the latter equality extends to disconnected strongly regular graphs.
\end{proof}

In the context of Lemma~\ref{lemma: energy of strongly regular graphs}, it is worth noting the following
result on maximal energy graphs, which combines Theorems~5.9 and~5.10 of \cite{LiSG12}.
\begin{theorem}[Maximal energy graphs, \cite{Haemers08,KoolenM01}]
\label{theorem: maximal energy graphs}
{\em Let $\Gr{G}$ be a graph on $n$ vertices. Then,
\begin{align}
\label{eq: UB on graph energy}
\Enr{\Gr{G}} \leq \tfrac12 n (\sqrt{n}+1),
\end{align}
with equality if $\Gr{G}$ is a strongly regular graph in the family
\begin{align}
\label{eq: srg}
\SRG\Bigl(n, \tfrac12 (n+\sqrt{n}), \tfrac14 (n+2\sqrt{n}), \tfrac14 (n+2\sqrt{n}) \Bigr).
\end{align}
Furthermore, such strongly regular graphs exist if one of the following conditions hold:
\begin{enumerate}[(1)]
\item $n = 4^p q^4$ with $p,q \in \naturals$;
\item $n = 4^{p+1} q^2$ with $p \in \naturals$ and $4q-1$ is a prime power, or $2q-1$ is a prime
power, or $q$ is a square of an integer, or $q<167$.
\end{enumerate}}
\end{theorem}

\begin{corollary}
\label{corollary: induced subgraph of srg}
{\em Let $\Gr{H}$ be an induced subgraph of a strongly regular graph $\Gr{G}$ belonging to the family $\srg{n}{d}{\lambda}{\mu}$.
Then, the following inequalities hold:
\begin{enumerate}[(1)]
\item \label{part 1: induced subgraph of srg}
\begin{align}
\label{eq1a: induced subgraph of srg}
& \vartheta(\Gr{H}) \leq \dfrac{n \, (t+\mu-\lambda)}{2d+t+\mu-\lambda}, \\[0.1cm]
\label{eq2a: induced subgraph of srg}
& \Enr{\Gr{H}} \leq d \, \biggl( 1 + \frac{2(n-d)+\lambda-\mu}{t} \biggr),
\end{align}
where $t \eqdef \sqrt{(\lambda-\mu)^2 + 4(d-\mu)}$, as given in \eqref{eq3: 15.02.25}.
\item \label{part 2: induced subgraph of srg}
Specifically, if $\Gr{H}$ is in the family $\srg{n'}{d'}{\lambda'}{\mu'}$, where $n' \leq n$, $d' \leq d$, $\lambda' \leq \lambda$,
and $\nu' \leq \nu$, then
\begin{align}
\label{eq1b: induced subgraph of srg}
& \dfrac{n'}{n} \cdot \dfrac{t'+\mu'-\lambda'}{t+\mu-\lambda} \cdot \dfrac{2d+t+\mu-\lambda}{2d'+t'+\mu'-\lambda'} \leq 1, \\[0.1cm]
\label{eq2b: induced subgraph of srg}
& \frac{t d'}{t' d} \cdot \frac{t'+2(n'-d')+\lambda'-\mu'}{t+2(n-d)+\lambda-\mu} \leq 1,
\end{align}
where $t' \eqdef \sqrt{(\lambda'-\mu')^2 + 4(d'-\mu')}$.
\end{enumerate}}
\end{corollary}
\begin{proof}
Inequalities \eqref{eq1a: induced subgraph of srg}--\eqref{eq2b: induced subgraph of srg} follow from
Lemma~\ref{lemma: induced subgraph}, combined with its specialization to strongly regular graphs according
to Theorem~\ref{theorem: theta of srg} and Lemma~\ref{lemma: energy of strongly regular graphs}.
\end{proof}

\begin{remark}
{\em The computational complexity of the Lov\'{a}sz $\vartheta$-function for a graph $\Gr{G}$ on $n$ vertices,
obtained by numerically solving the SDP problem in \eqref{eq: SDP}, is polynomial in $n$ and $\log \frac1{\varepsilon}$,
where $\varepsilon$ denotes the computation precision (see Theorem~11.11 of \cite{Lovasz19}). The constrained
convex optimization problem in \eqref{eq: SDP} involves $\tfrac12 n(n+1)$ optimization variables -- specifically, the entries
of the $n \times n$ positive semidefinite matrix ${\bf{B}}$ -- and at most quadratically many equality constraints in $n$
(scaling linearly in $n$ for sparse graphs). Consequently, while the Lov\'{a}sz $\vartheta$-function
remains computationally feasible for graphs with up to several hundred vertices, it becomes impractical for significantly
larger instances.
Checking the satisfiability of the necessary conditions for spanning or induced subgraphs in Lemmata~\ref{lemma: spanning subgraph}
and~\ref{lemma: induced subgraph} is thus infeasible for large $n$. However, computing only the second-largest or smallest
eigenvalue of the adjacency matrix can be efficiently performed using iterative methods, which avoid computing the entire spectrum.
These methods are particularly well-suited for large graphs, enabling the verification of the necessary conditions for spanning
or induced subgraphs of regular graphs in Corollary~\ref{corollary: 26.02.2025} and Item~1 of Corollary~\ref{corollary: induced subgraph of srg},
even for large instances. Furthermore, verifying the necessary conditions in Corollary~\ref{corollary: necessary cond.} and Item~2 of
Corollary~\ref{corollary: induced subgraph of srg} for spanning or induced strongly regular subgraphs of strongly regular graphs is
computationally straightforward and imposes no practical limitations, even for exceedingly large strongly regular graphs.}
\end{remark}

\begin{example}[Cycles as induced subgraphs of the Shrikhande Graph]
{\em The Shrikhande graph, denoted by $\Gr{G}$, is a strongly regular graph in the family $\srg{16}{6}{2}{2}$ (there are
two such nonisomorphic strongly regular graphs, see Section~10.6 of \cite{BrouwerM22}).
The graph $\Gr{G}$ is thus a strongly regular graph in which every pair of vertices has exactly two common neighbors (by the friendship theorem,
it stays in contrast to the non-existence of a regular graph, except for a triangle, in which every pair of vertices has exactly
one common neighbor). It
can be verified that $\Gr{G}$ contains induced cycles of length~3, 4, 5, 6, and~8, but no induced cycles of length~7 or greater
than~8. Figure~\ref{figure: Shrikhande graph} illustrates two such induced cycles, with a cycle of length~6 shown on the left plot
and a cycle of length~8 on the right plot, both generated using the SageMath software \cite{SageMath}.
\begin{figure}[htb]
\includegraphics[width=8cm]{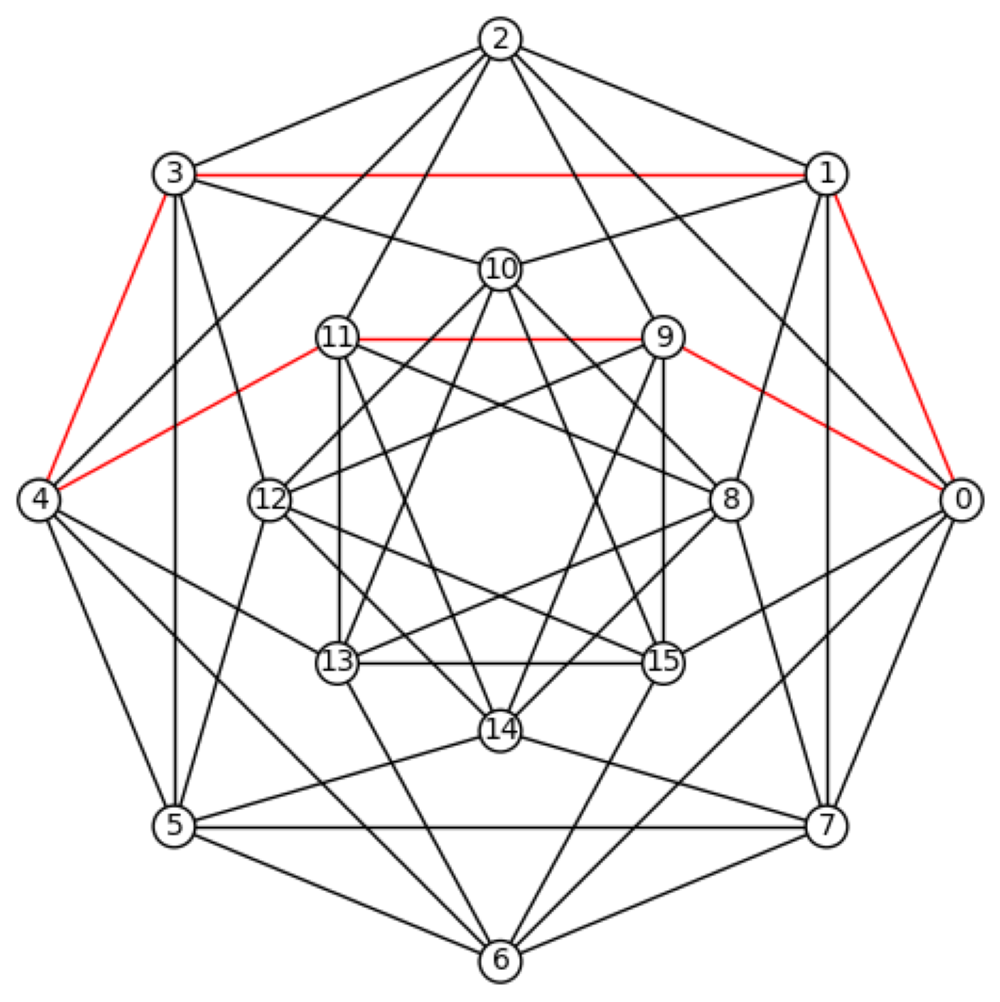}
\includegraphics[width=8cm]{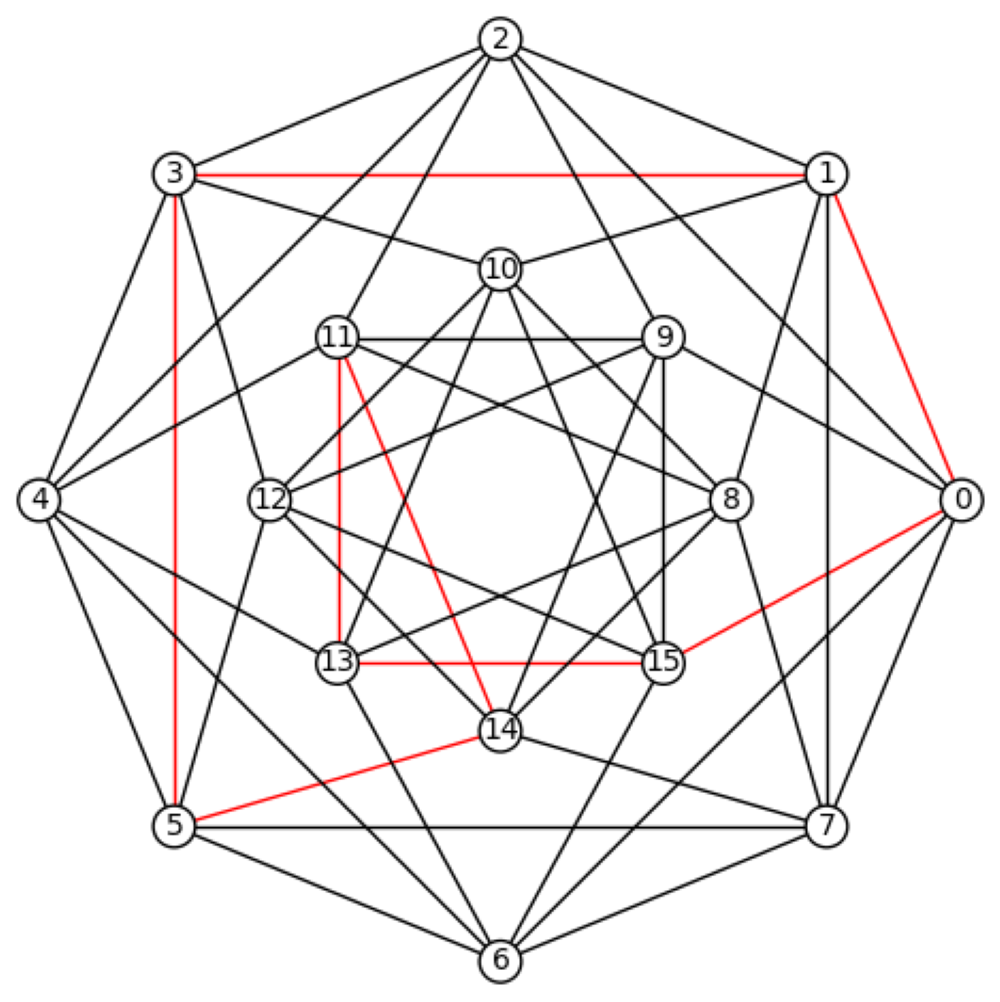}
\caption{The Shrikhande graph $\Gr{G}$ and two induced cycles of length~6 (left plot) and~8 (right plot), with their edges highlighted
in red. \label{figure: Shrikhande graph}}
\end{figure}
Let $\Gr{H} = \CG{\ell}$ denote a cycle of length $\ell \geq 3$. The necessary conditions in \eqref{eq1a: induced subgraph of srg} and
\eqref{eq2a: induced subgraph of srg} for a graph isomorphic to $\CG{\ell}$ to be an induced subgraph of $\Gr{G}$ are given by
$\vartheta(\CG{\ell}) \leq 4$ and $\Enr{\CG{\ell}} \leq 36$, respectively. These follow from the equalities $\vartheta(\Gr{G})=4$ by
Theorem~\ref{theorem: theta of srg} and $\Enr{\Gr{G}} = 36$ by Lemma~\ref{lemma: energy of strongly regular graphs}.
It can be verified that the former condition on the Lov\'{a}sz $\vartheta$-function is more restrictive in the studied case than the latter condition on the
graph energy, so we focus on the necessary condition $\vartheta(\CG{\ell}) \leq 4$. By \cite{Lovasz79_IT}, for all integers $\ell \geq 3$,
\begin{align}
\label{eq: theta of cycle graphs}
\vartheta(\CG{\ell}) =
\begin{dcases}
\frac{\ell}{1 + \sec \frac{\pi}{\ell}}, & \quad \ell \in \{3, 5, 7, \ldots \}, \\
\frac{\ell}{2}, & \quad \ell \in \{4, 6, 8, \ldots \}.
\end{dcases}
\end{align}
Consequently, by \eqref{eq: theta of cycle graphs}, for all $\ell > 8$, we get $\vartheta(\CG{\ell}) \geq \vartheta(\CG{9}) = 4.361 \ldots > 4 = \vartheta(\Gr{G})$,
verifying that $\Gr{G}$ does not contain any induced cycle of length $\ell > 8$.  Furthermore, $\vartheta(\CG{8}) = 4 = \vartheta(\Gr{G})$
and $\Gr{G}$ contains a cycle of length~8 as an induced subgraph (see the right plot of Figure~\ref{figure: Shrikhande graph}). This provides an example where
the necessary condition holds with equality, coinciding with the existence of the longest possible induced cycle. However, since $\Gr{G}$ does not contain an induced
cycle of length~7, despite $\vartheta(\CG{7}) = 3.31766 \ldots < 4 = \vartheta(\Gr{G})$, this serves as a counterexample where the necessary conditions in Item~1 of Corollary~\ref{corollary: induced subgraph of srg} fail to exclude the existence of an induced cycle of length~7 in $\Gr{G}$.

Another counterexample where the necessary conditions in Item~1 of Corollary~\ref{corollary: induced subgraph of srg} fail to exclude the existence of an induced cycle
of a given length is a connected, triangle-free graph $\Gr{G}$ of energy at least~4 (noting that $\Enr{\CG{3}} = 4$). By definition, such a graph does not contain any
triangle as a subgraph, yet the considered necessary conditions do not rule out the absence of an induced triangle in $\Gr{G}$. Among the connected strongly regular
graphs, there are seven currently known triangle-free graphs, and all of them are determined by their adjacency spectrum (see, e.g., Section~4.7 of \cite{SasonKHB25}).}
\end{example}

\begin{example}[Induced subgraphs of triangular graphs]
\label{example: induced subgraphs of triangular graphs}
{\em The family of triangular graphs consists of regular graphs obtained as the line graphs of complete
graphs on at least three vertices. Let $T_\ell$ denote the line graph of the complete graph $\CoG{\ell}$,
for all $\ell \geq 3$. This results in a strongly regular graph in the family
$\srg{\tfrac12 \ell(\ell-1)}{2\ell-4}{\ell-2}{4}$ for all $\ell \geq 4$ (see Section~1.1.7 of \cite{BrouwerM22}).
For $\ell \geq 4$ with $\ell \neq 8$, these connected, strongly regular graphs are unique for their respective
parameters. Their uniqueness further implies that they are determined by their spectra (following from
Theorem~34 of \cite{SasonKHB25}).

Now, let $\Gr{H}$ be a strongly regular graph in the family $\srg{n'}{d'}{\lambda'}{\mu'}$
for some specified (feasible) parameters. We consider whether $\Gr{H}$ can appear as an induced subgraph
of the triangular graph $T_\ell$, as a function of $\ell \geq 4$. Applying Item~2
of Corollary~\ref{corollary: induced subgraph of srg}, if either inequality \eqref{eq1b: induced subgraph of srg}
or \eqref{eq2b: induced subgraph of srg} is violated, then no strongly regular graph in the family
$\srg{n'}{d'}{\lambda'}{\mu'}$ (including $\Gr{H}$) can be an induced subgraph of $T_\ell$. This leads to a general
conclusion for those values of $\ell$ where either inequality fails. Applying Item~2
of Corollary~\ref{corollary: induced subgraph of srg} to $T_\ell$ with the parameters:
\begin{align}
n = \tfrac12 \ell (\ell-1), \quad d = 2(\ell-2), \quad \lambda = \ell-2, \quad \mu=4,
\end{align}
yields, by \eqref{eq3: 15.02.25},
\begin{align}
t = \ell-2, \quad \forall \, \ell \geq 4.
\end{align}
With straightforward algebra, the inequalities \eqref{eq1b: induced subgraph of srg} and \eqref{eq2b: induced subgraph of srg}
simplify to
\begin{align}
\label{eq1: 19.02.25}
\begin{dcases}
\dfrac{2n' (t'+\mu'-\lambda')}{2d'+t'+\mu'-\lambda'} \leq \ell, \\[0.1cm]
\dfrac{d' \, (t'+2(n'-d')+\lambda'-\mu')}{t'} \leq 2 \ell (\ell-3),
\end{dcases}
\end{align}
where $t' \eqdef \sqrt{(\lambda'-\mu')^2 + 4(d'-\mu')}$. If, for given parameters $(n', d', \lambda', \mu')$, at
least one of these inequalities is violated for a specific $\ell \geq 4$, then $\Gr{H}$ cannot be an induced subgraph of
$T_\ell$.}
\end{example}

\begin{example}[Continuation to Example~\ref{example: induced subgraphs of triangular graphs}: the Gewirtz graph]
\label{example: Gewirtz graph}
{\em Consider the Gewirtz graph, denoted by $\Gr{H}$, which is a connected, triangle-free, strongly regular graph in
the family $\srg{56}{10}{0}{2}$ (see Section~10.20 of \cite{BrouwerM22}). To clarify, since $\mu = 2 > 0$, the graph
$\Gr{H}$ is connected, and $\lambda=0$ confirms that it is triangle-free. Uniqueness in this parameter family ensures it is
determined by its spectrum \cite{SasonKHB25}.
Specializing \eqref{eq1: 19.02.25} for $\Gr{H}$, the first inequality is violated if and only if $\ell \leq 31$
(it dominates the second, energy-based, inequality in \eqref{eq1: 19.02.25}). Hence, $\Gr{H}$ cannot be an induced subgraph
of $T_\ell$ for all $\ell \leq 31$.
For $\ell \leq 11$, this result is trivial since $\tfrac12 \ell(\ell-1) < 56$, meaning that the order of
$T_\ell$ is too small to contain $\Gr{H}$ as a subgraph. However, for $\ell \geq 32$, this question remains open,
as both inequalities in \eqref{eq1: 19.02.25} hold}.
\end{example}


\begin{thebibliography}{99}

\bibitem{ErdosRS66}
P. Erd\H{o}s, A. R\'{e}nyi, and V. S\'{o}s, ``A problem in graph theory,'' {\em Studia Scieutiaruni
Mathernaticarum Huugarica}, vol.~1, pp.~215--235, 1966.
\url{https://static.renyi.hu/renyi_cikkek/1966_on_a_problem_of_graph_theory.pdf}
(accessed on February 28, 2025).

\bibitem{AignerZ18}
M. Aigner, and G. M. Ziegler, {\em Proofs from the book}, Sixth Edition, Springer, Berlin, Germany, 2018.
Available from: \url{https://link.springer.com/book/10.1007/978-3-662-57265-8} (accessed on February 28, 2025)

\bibitem{MertziosU16}
G. B. Mertzio and W. Unger, ``The friendship problem on graphs,'' {\em Journal of Multiple-Valued Logic \&
Soft Computing}, vol.~27, pp.~275--285, 2016. \url{https://mertzios.net/papers/Jour/Jour_Friendship-Problem.pdf}
(accessed on February 28, 2025).

\bibitem{Sason23}
I. Sason, ``Observations on the Lov\'{a}sz $\vartheta$-function, graph capacity, eigenvalues, and strong products,''
{\em Entropy}, vol.~25, no.~1, paper~104, pp.~1--40, January~2023. \url{https://doi.org/10.3390/e25010104}

\bibitem{Lovasz79_IT}
L. Lov\'{a}sz, ``On the Shannon capacity of a graph,'' {\em IEEE Transactions on Information Theory},
vol.~25, no.~1, pp.~1--7, January 1979. \url{https://doi.org/10.1109/TIT.1979.1055985}

\bibitem{Lovasz19}
L. Lov\'{a}sz, {\em Graphs and Geometry}, American Mathematical Society, volume~65, 2019.
\url{https://doi.org/10.1090/coll/065}

\bibitem{Brouwer}
A. E. Brouwer, Tables of paramters of strongly regular graphs. \url{https://aeb.win.tue.nl/graphs/srg/}
(accessed on February 28, 2025).

\bibitem{Maksimovic18}
M. Maksimovi\'{c}, ``Enumeration of strongly regular graphs on up to 50~vertices having $S_3$ as an automorphism group,''
{\em Symmetry}, vol.~10, no.~6, paper~212, June~2018. \url{https://doi.org/10.3390/sym10060212}

\bibitem{BrouwerM22}
A. E. Brouwer and H. van Maldeghem, {\em Strongly Regular Graphs}, Cambridge University Press, (Encyclopedia
of Mathematics and its Applications, Series Number 182), 2022. \url{https://doi.org/10.1017/9781009057226}

\bibitem{Gritsenko01}
O. Gritsenko, ``On strongly regular graph with parameters $(65, 32, 15, 16)$,'' February 10, 2021.
\url{https://arxiv.org/abs/2102.05432}

\bibitem{vanLintW02}
J. H. van Lint and R. M. Wilson, {\em A Course in Combinatorics}, second edition, Cambridge University Press, 2001.
\url{https://doi.org/10.1017/CBO9780511987045}

\bibitem{ChoiCK25}
M. Choi, H. Chu, and S. R. Kim, ``A digraph version of the friendship theorem,'' {\em Discrete Mathematics},
vol.~348, no.~1, pp.~1--6, January 2025. \url{https://doi.org/10.1016/j.disc.2024.114238}

\bibitem{SageMath}
{\it The Sage Developers}, SageMath, the Sage Mathematics Software System, Version 9.3, 2021.
\url{https://www.sagemath.org/} (accessed on February 28, 2025).

\bibitem{Sason24}
I. Sason, ``Observations on graph invariants with the Lov\'{a}sz $\vartheta$-function,'' {\em AIMS Mathematics},
vol.~9, no.~6, pp.~15385--15468, April 2024. \url{https://doi.org/10.3934/math.2024747}

\bibitem{Gutman78}
I. Gutman, ``The energy of a graph,'' {\em Steiermärkisches Mathematisches Symposium (Stift Rein, Graz, 1978)},
no.~103, pp.~1--22, 1978.

\bibitem{LiSG12}
X. Li, Y. Shi, and I. Gutman, {\em Graph Energy}, Springer, 2012.
\url{https://doi.org/10.1007/978-1-4614-4220-2}


\bibitem{BrouwerH2011}
A. E. Brouwer and W. H. Haemers, {\em Spectra of Graphs}, Springer Science \& Business Media, 2011.
\url{https://doi.org/10.1007/978-1-4614-1939-6}

\bibitem{Haemers08}
W. H. Haemers, ``Strongly regular graphs with maximal energy,'' {\em Linear Algebra and Its Applications},
vol~429, no.~11--12, December 2008. \url{https://doi.org/10.1016/j.laa.2008.03.024}

\bibitem{KoolenM01}
J. H. Koolen and V. Moulton, ``Maximal energy graphs,'' {\em Advances in Applied Mathematics},
vol.~26, no.~1, pp.~47--52, January 2001. \url{https://doi.org/10.1006/aama.2000.0705}

\bibitem{SasonKHB25}
I. Sason, N. Krupnik, S. Hamud, and A. Berman, ``On spectral graph determination,'' {\em Mathematics}, vol.~13,
no.~4, paper~549, pp.~1--44, February 2025. \url{ https://doi.org/10.3390/math13040549}

\end{thebibliography}
\end{document}